%% file: main.tex
\documentclass[11pt]{amsart}
\pdfoutput=1

\pdfpageattr{/Group <</S /Transparency /I true /CS /DeviceRGB>>}

\usepackage[T1]{fontenc}
\usepackage[english]{babel}
\usepackage{geometry}
\usepackage{amsfonts,amsmath,amssymb,amsthm,mathrsfs}
\usepackage{tikz}
\usetikzlibrary{calc}
\usepackage{caption}
\usepackage{subcaption}
\usepackage{xspace}
\usepackage{booktabs,tabularx}
\usepackage{dcolumn}
\usepackage{marvosym}
\usepackage{microtype}
\usepackage{enumerate}
\usepackage{mathtools}
\usepackage{hyperref}
\usepackage[autostyle=true]{csquotes}

\microtypesetup{tracking,kerning,spacing}
\microtypecontext{spacing=nonfrench}

\hypersetup{colorlinks,linkcolor=blue,citecolor=blue,urlcolor=blue}

\newcommand\cprime{$'$}

\newcommand\polymake{{\tt polymake}\xspace}
\newcommand\mptopcom{{\tt mptopcom}\xspace}

\newcommand\NN{{\mathbb N}}

\newcommand\RR{{\mathbb R}}
\newcommand\ZZ{{\mathbb Z}}

\newcommand\cF{\mathcal{F}^\Sigma}
\newcommand\cP{\mathcal{P}}
\newcommand\cQ{\mathcal{Q}}

\newcommand\SetOf[2]{\left\{\left.#1\vphantom{#2}\ \right|\ #2\vphantom{#1}\right\}}
\newcommand\smallSetOf[2]{\{{#1}\,|\,{#2}\}}
\DeclareMathOperator{\conv}{conv}

\DeclareMathOperator{\relint}{relint}
\DeclareMathOperator{\rank}{rk}

\DeclareMathOperator{\Dr}{Dr} 
\DeclareMathOperator{\size}{\#}
\DeclareMathOperator{\aff}{aff}

\DeclareMathOperator{\fig}{VF}
\DeclareMathOperator{\Sym}{Sym}
\newcommand{\relation}{\preceq_P\,}
\newcommand{\relationQ}{\preceq_Q\,}
\DeclareMathOperator{\pos}{pos}
\DeclareMathOperator{\matroid}{M}
\DeclareMathOperator{\polytope}{P}

\theoremstyle{plain}
    \newtheorem{theorem}{Theorem}
    \newtheorem{corollary}[theorem]{Corollary}
    \newtheorem{lemma}[theorem]{Lemma}
    \newtheorem{proposition}[theorem]{Proposition}
\theoremstyle{definition}
    \newtheorem{remark}[theorem]{Remark}
    \newtheorem{example}[theorem]{Example}

\title[Multi-splits and tropical linear spaces from nested matroids]{Multi-splits and tropical linear spaces\\ from nested matroids}

\author{Benjamin Schr\"oter}

\address{
  Institut f{\"u}r Mathematik,
  TU Berlin,
  Str.\ des 17. Juni 136, 10623 Berlin, Germany
}
\email{schroeter@math.tu-berlin.de}


\begin{document}

\begin{abstract}
   In this paper we present an explicit combinatorial description of a special class of facets of the secondary polytopes of hypersimplices.
   These facets correspond to polytopal subdivisions called multi-splits.
   We show  a relation between the cells in a multi-split of the hypersimplex and nested matroids.
   Moreover, we get a description of all multi-splits of a product of simplices.
   Additionally, we present a computational result to derive explicit lower bounds on the number of facets of secondary polytopes of hypersimplices.
\end{abstract}
\maketitle

\section{Introduction}
\noindent
It is a natural idea to decompose a difficult problem into smaller pieces.
There are many natural situations in which one has fixed a finite set of points, i.e, a \emph{point configuration}.
All convex combinations of these points form a convex body called \emph{polytope}.
For a basic background on polytopes see the monograph \cite{Ziegler:2000} by Ziegler.
It is typical to ask for specific subdivisions or even all subdivisions of a polytope into smaller polytopes whose vertices are points of a given point configuration.
The given points are often the vertices of the polytope.
Famous examples for subdivisions are placing, minimum weight, Delaunay triangulations and regular subdivisions in general. 
For an overview of applications see the monograph \cite{LoeraRambauSantos:2010} by De Loera, Rambau, and Santos.

All subdivisions form a finite lattice with respect to coarsening and refinement.
Gel{\cprime}fand, Kapranov and Zelevinsky showed that the sublattice of regular subdivisions is the face-lattice of a polytope; see \cite{GelfandKapranovZelevinsky:1994}.
This polytope is called \emph{secondary polytope} of the subdivision.
The vertices of the secondary polytope correspond to finest subdivisions, i.e., triangulations.
This polytope can be realized as the convex hull of the GKZ-vectors. 
An important example in combinatorics is the \emph{associahedron}, which is the secondary polytope of a convex $n$-gon; see \cite{CeballosSantosZiegler:2015}.
It is remarkable that the number of triangulations of an $n$-gon is the Catalan number $\frac{1}{n-1}\tbinom{2n-4}{2n-2}$ and the number of diagonals is $\frac{n(n-3)}{2}$, a triangular number minus one.
A subdivision into two maximal cells is a coarsest subdivision and called \emph{split}.
The coarsest subdivisions of the $n$-gon are the splits along the diagonals.
This  example shows that the associahedron has $\frac{1}{n-1}\tbinom{2n-4}{2n-2}$ vertices and only $\frac{n(n-3)}{2}$ facets.
It is expectable that in general the number of facets of the secondary polytope is much smaller than the number of vertices.

Herrmann and Joswig were the first who systematically studied splits and hence facets of the secondary polytope. Herrmann introduced a generalization of splits in \cite{Herrmann:2011}.
A multi-split is a coarsest subdivision, such that all maximal cells meet in a common cell.

The purpose of this paper is to further explore the facet structure of the secondary polytope for two important classes of polytopes --  products of simplices and hypersimplices.
In particular, we investigate their multi-splits.
Triangulations of products of simplices have been studied in algebraic geometry, optimization and game theory; see \cite[Section 6.2]{LoeraRambauSantos:2010}.
An additional motivation to study splits of products of simplices is their relation to tropical convexity \cite{DevelinSturmfels:2004}, tropical geometry and matroid theory.

The focus of our interest is on hypersimplices.
The \emph{hypersimplex} $\Delta(d,n)$ is the slice of the $n$-dimensional $0/1$-cube with the hyperplane $x_1+\ldots+x_n=d$.
Hypersimplices appear frequently in mathematics.
For example, they appear in algebraic combinatorics, graph theory, optimization, analysis and number theory (see \cite[Subsection 6.3.6]{LoeraRambauSantos:2010}), as well as in phylogenetics, matroid theory and tropical geometry.
The latter three topics are closely related, and splits of hypersimplices play an important role in all of them.
Bandelt and Dress \cite{BandeltDress:1992} were the first who studied the split decomposition of a finite metric in  phylogenetic analysis.
Later Hirai \cite{Hirai:2006},  Herrmann, Joswig \cite{HerrmannJoswig:2008} and Koichi \cite{Koichi:2014} developed split decompositions of polyhedral subdivisions. 
In particular they discussed subdivisions of hypersimplices.
The special case of a subdivision of a hypersimplex $\Delta(2,n)$ corresponds to a class of finite pseudo-metrics.
While the matroid subdivisions of $\Delta(2,n)$ are totally split-decomposable and correspond to phylogenetic trees with $n$ labeled leaves; see \cite{HerrmannJoswig:2008} and \cite{SpeyerSturmfels:2004}.

A product of simplices appears as vertex figures of any vertex of a hypersimplex.
Moreover, a subdivision of a product of simplices extends to a subdivision of a hypersimplex via the tropical Stiefel map.
This lift has been studied in \cite{HerrmannJoswigSpeyer:2014}, \cite{Rincon:2013} and \cite{FinkRincon:2015}.

This paper comprises three main results, that combine polyhedral and matroid theory as well as tropical geometry.
In Section~\ref{sec:splits_hypersimplex} we show that any multi-split of a hypersimplex is the image of a multi-split of a product of simplices under the tropical Stiefel map (Theorem~\ref{thm:main1}).
To reach this goal we introduce the concept of \enquote{negligible} points in a point configuration.
With this tool we are able to show that the point configuration consisting of the vertices of a product of simplices suffice to describe a given multi-split of the hypersimplex.
This already implies that all multi-splits of a hypersimplex are subdivisions into matroid polytopes.

In Section~\ref{sec:nested} we define a relation depending on matroid properties of the occurring cells.
We use this relation to enumerate all multi-splits of hypersimplices (Proposition~\ref{prop:enum_ksplits}) and show that all maximal cells in a multi-split of a hypersimplex correspond to matroid polytopes of nested matroids (Theorem~\ref{thm:main2}). This generalizes the last statement of \cite[Proposition 30]{JoswigSchroeter:2017}, which treats $2$-splits, i.e. multi-splits with exactly two maximal cells.
As a consequence of the enumeration of all multi-splits of a hypersimplex we get the enumeration of all multi-splits of a product of simplices (Theorem~\ref{thm:simplices}).
Nested matroids are a well studied class in matroid theory.
Hampe recently introduced the \enquote{intersection ring of matroids} in \cite{Hampe:2017} and showed that every matroid is a linear combination of nested matroids in this ring.
Moreover, matroid polytopes of nested matroids describe the intersection of linear hyperplanes in a matroid subdivision locally. Hence they occur frequently in those subdivisions; see \cite{JoswigSchroeter:2017}.

In the last Section~\ref{sec:computations} we take a closer look at coarsest matroid subdivisions of the hypersimplex in general. 
Matroid subdivisions are important in tropical geometry  as they are dual to tropical linear spaces.
If they are regular they are also called \enquote{valuated matroids} introduced by Dress and Wenzel \cite{DressWenzel:1992}.
Coarsest matroid subdivisions have been studied in \cite{HerrmannJoswigSpeyer:2014}.
We compare two constructions of matroid subdivisions, those that are in the image of the tropical Stiefel map and those that appear as a corank vector of a matroid.
We present our computational results on the number of coarsest matroid subdivisions of the hypersimplex $\Delta(d,n)$ for small parameters $d$ and $n$ (Proposition~\ref{prop:comp}), which illustrate how fast the number of combinatorial types of matroid subdivisions grows.

\section{Multi-splits of the hypersimplex}\label{sec:splits_hypersimplex}
\noindent
In this section we will study a natural class of coarsest subdivisions, called \enquote{multi-splits}.
Our goal is to show that any \enquote{multi-split} of the hypersimplex can be derived from a \enquote{multi-split} of a product of simplices.
We assume that the reader has a basic background on subdivisions and secondary fans.
The basics could be found in \cite{LoeraRambauSantos:2010}.
We will shortly introduce our notation and definitions.

We consider a finite set of points in $\RR^n$ as a \emph{point configuration} $\cP$, i.e., each point occurs once in $\cP$.
A subdivision $\Sigma$ of $\cP$ is a collection of subsets of $\cP$, such that they satisfy the Closure, Union and Intersection Property. We call the convex hull of such a subset a \emph{cell}.
The \emph{lower convex hull} of a polytope $Q\subset\RR^{n+1}$ is the collection of all faces with an inner facet normal with a strictly positive $(n+1)$-coordinate.
A subdivision is \emph{regular}  when it is combinatorially isomorphic to the lower convex hull of a polytope $Q\subset\RR^{n+1}$, this polytope is called the \emph{lifted polytope}.
The $(n+1)$-coordinate is called the \emph{height}.
The heights of the points in $\cP$ form the \emph{lifting vector}.
The set of all lifting vectors whose projection of the lower convex hull coincides form an open cone.
The closure of such a cone is called a \emph{secondary cone}. 
The collection of all secondary cones is the \emph{secondary fan} of the point configuration $\cP$. 
We call a point $q\in\cP$ \emph{negligible in the subdivision $\Sigma$} if there is a cell containing the point $q$ and $q$ does not occur as a vertex of any $2$-dimensional cell.
In particular, a negligible point $q$ lies in a cell $C$ if and only if $q \in \conv(C\setminus \{q\})$.
For a regular subdivision this means that $q$ is lifted to a redundant point and to the lower convex hull of the lifted polytope.

\begin{example}\label{ex:fivepoints}
Consider the point configuration of the following five points $(0,0),(3,0),(0,3),$ $(3,3),(1,1)$. 
All nine possible subdivisions of that point configuration are regular and $(1,1)$ is negligible in three of them.
   See Figure~\ref{fig:fivepoints}.
\end{example}

\begin{figure}[t]
   \input{example1.tikz}
   \caption{Nine (regular) subdivisions of the five points of Example~\ref{ex:fivepoints}. The inner point is negligible in all subdivisions in the middle row. This point is lifted above the lower convex hull in the regular subdivisions in the top row.
   The subdivision in the middle of the top row is a $1$-split, the left and right in the second row are $2$-splits and in the middle of the bottom row is a $3$-split.}
  \label{fig:fivepoints}
\end{figure}
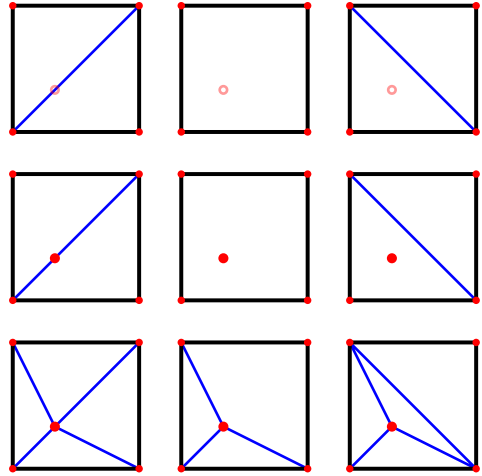

A negligible point $q\in\cP$ can be omitted in the subdivision $\Sigma$. More precisely we have the following relation between the subdivisions of $\cP$ and 
those subdivisions of $\cP\setminus\{q\}$.

\begin{proposition}\label{prop:negible} Let $q\in\cP$, sucht that $q\in\conv(C\setminus \{q\})$. Consider the following map on the set of all subdivisions of $\cP$ where $q$ is negligible.
\[
 \Sigma \mapsto \SetOf{C\setminus\{q\}}{C\in\Sigma} 
\]
 This map is a bijection onto all subdivisions of $\cP\setminus\{q\}$.
\end{proposition}
 
A \emph{$k$-split} of a point configuration $\cP$ is a coarsest subdivision $\Sigma$ of the convex hull $P$ of $\cP$ with $k$ maximal faces and a common $k-1$-codimensional face.
We call this face the \emph{common cell} and denote this polytope by $H^\Sigma$. 
We shorten the notation if the point configuration $\cP$ is the vertex set of a polytope $P$ and write this as $k$-split of $P$.
If we do not specify the number of maximal cells we will call such a coarsest subdivision a \emph{multi-split}. 
\begin{example} The point configuration of the points in Example~\ref{ex:fivepoints} has four coarsest subdivisions. These are a $1$-split, two $2$-splits and a $3$-split. See Figure~\ref{fig:fivepoints}.
\end{example}

\begin{example} In general not all coarsest subdivisions are multi-splits.
   An extremal example is a $4$-dimensional cross polytope with perturbed vertices, such that four points do not lie in a common hyperplane.
   The secondary polytope of this polytope has $29$ facets, non of which is a multi-split.
\end{example}

\begin{example} Another example for a coarsest subdivision that is not a multi-split is illustrated in Figure~\ref{fig:nonksplit}.
\end{example}

Splits have been studied by several people in phylogenetic analysis, metric spaces and polyhedral geometry. For example by Bandelt and Dress \cite{BandeltDress:1992}, Hirai \cite{Hirai:2006}, Herrmann and Joswig \cite{HerrmannJoswig:2008} and by Koichi \cite{Koichi:2014}. The more general multi-splits have been introduced by Herrmann in \cite{Herrmann:2011} under the term $k$-split.
The main result there is the following.
\begin{proposition}[{\cite[Theorem 4.9]{Herrmann:2011}}]\label{prop:ksplitHer}
   Each $k$-split is a regular subdivision. The dual complex of the lower cells, i.e., the subcomplex in the polar of the lifted polytope, is a $k$-simplex modulo its lineality space.
\end{proposition}
Proposition~\ref{prop:ksplitHer} implies that the subdivision of a multi-split corresponds to a ray in the secondary fan, i.e., this is a coarsest regular and non-trivial subdivision. 
Furthermore, the number of faces of a fixed dimension of a $k$-split is the same as the number of faces of the same dimension of a $k$-simplex. In particular, the number of maximal, non-trivial inclusionwise minimal cells equals $k$.
Note that also the number of $(n-k+2)$-dimensional cells is equal to $k$.

We recall the main construction of Proposition~\ref{prop:ksplitHer}, which proves the regularity.
The subdivision $\Sigma$ is induced by a complete fan $\cF$ with $k$ maximal cones, a lineality space $\aff H^\Sigma$ and an apex at $a\in\RR^n$.
Here \enquote{induced} means a cell of $\Sigma$ is the intersection of a cone of $\cF$ with $P$.
The apex $a$ is not unique, it can be any point in $H^\Sigma$. Later we will take specific choices for it.
A lifting function that induces the multi-split is given by the following.
All points in $\cP \cap \aff H^\Sigma$ are lifted to height zero. 
The height of a point $p\in\cP$ that is contained in a ray of $\cF$ is the shortest distance to the affine space $\aff H^\Sigma$.
Each other point in the point configuration $\cP$ is a non-negative linear combination of those rays.
The height of a point is given by the linear combination with the same coefficients multiplied with the heights of points in the rays of $\cF$.

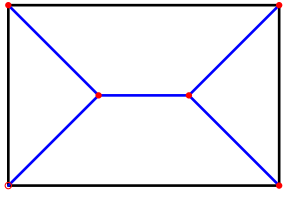
\begin{figure}[t] 
   \input{nonksplit.tikz}
   \caption{A coarsest regular subdivision, which is not a multi-split.}
  \label{fig:nonksplit}
\end{figure}

The following Lemma summarizes important properties of the common cell $H^\Sigma$.
\begin{lemma}\label{lem:propertiesH}
   The common cell $H^\Sigma$ is the intersection of the affine space $\aff H^\Sigma$ with $P$ and $\aff H^\Sigma$ intersects $P$ in its relative interior. Hence, the relative interior of the common cell $H^\Sigma$ is contained in the relative interior of the polytope $P$.
\end{lemma}
\begin{proof}
   Let us assume that $\cF$ is the complete fan of the $k$-split $\Sigma$.
   The intersection of all maximal cones in $\cF$ is an affine space which shows $H^\Sigma = \aff H^\Sigma \cap P$.
   The dual cell of $H^\Sigma$ is a $k$-simplex by Proposition~\ref{prop:ksplitHer}, and therefore a bounded polytope. 
Cells in the boundary of the polytope $P$ are dual to unbounded polyhedra.
Hence, this implies that $H^\Sigma$ is not contained in any proper face of $P$.
\end{proof}

Let $N(v)$ be the set of vertices that are neighbours of $v$ in the vertex-edge graph of $P$ and  $\varepsilon = \min_{u\in N(v)} \sum_{w\in N(v)} \langle w-v,\, u-v \rangle$.
The intersection of the polytope $P$ with a hyperplane that (weakly) separates the vertex $v$ from all other vertices and does not pass through $v$ is the \emph{vertex figure} of $v$
\[
   \fig(v) = \SetOf{x\in P }{ \sum_{w\in N(v)} \langle w-v,\, x-v \rangle = \varepsilon  }
\]
Our goal is to relate a $k$-split of a polytope to a $k$-split in a vertex figure.

We will focus on a particular class of convex polytopes, the hypersimplices.
We define for $d,n\in\ZZ$, $I\subseteq[n]$ and $0 \leq d \leq \size(I)$ the polytope
\[
   \Delta(d,I) \ = \ \SetOf{x\in[0,1]^n}{\sum_{i\in I} x_i = d \text{ and } \sum_{i\not\in I} x_i = 0 } \enspace .
\]
The \emph{$(d,n)$-hypersimplex} is the polytope $\Delta(d,[n])$, that we denote also by $\Delta(d,n)$.
Clearly, the polytope $\Delta(d,I)$ is a fixed embedding of the hypersimplex $\Delta(d,\size(I))$ into $n$-dimensional space.
We define the $(n-1)$-simplex $\Delta_{n-1}$ as the hypersimplex $\Delta(1,n)$ which is isomorphic to $\Delta(n-1,n)$.

\begin{example} \label{ex:vertex_fig}
   The vertex figure $\fig(e_I)$ of $e_I=\sum_{i\in I} e_i$ in the hypersimplex $\Delta(d,n)$ is
   \begin{align*}
      \fig(e_I) &= \SetOf{x\in\Delta(d,n)}{\sum_{i\in I} \sum_{j\in [n]-I} \langle e_j-e_i,\, x-e_I \rangle = n}\\
      &= \SetOf{x\in\Delta(d,n)}{\langle d e_{[n]-I}- (n-d) e_I,\, x-e_I \rangle = n}\\
      &= \Delta(d-1,I)\times\Delta(1,[n]-I)
   \end{align*}
\end{example}

If the point configuration $\cP$ is the vertex set of a polytope $P$, then there is at least one vertex that is contained in the common cell $H^\Sigma$.
Even more if $P$ is not $0$-dimensional then also $H^\Sigma$ is at least $1$-dimensional, otherwise it would be a face of $P$.

We say that a subdivision $\Sigma'$ on $\cP'\subsetneq\RR^n$ is \emph{ induced } by another subdivision $\Sigma$ on $\cP\subsetneq\RR^n$ if for all $\sigma\in\Sigma$ with $\dim(\conv\sigma\cap\conv \cP')=0$ we have $\conv\sigma\cap\conv(\cP')\subseteq \cP'$ and $\Sigma' = \smallSetOf{\conv\sigma\cap\cP'}{\sigma\in\Sigma}$.
Note that this is not the same concept as a subdivision that is \enquote{induced} by a fan.

\begin{figure}
   \begin{subfigure}[c]{0.49\textwidth}
   \input{oct.tikz}
      \subcaption{A $2$-split of the octahedron.}
   \end{subfigure}
   \begin{subfigure}[c]{0.49\textwidth}
   \input{vertex_figure.tikz}
      \subcaption{The induced $2$-split with the interior point $q$.}
   \end{subfigure}
  \caption{
     A $2$-split $\Sigma$ in the octahedron $\Delta(2,4)$, with the common cell \textcolor{blue}{$H^\Sigma$}
     and
     the induced $2$-split in the vertex figure $\fig(e_I)$.}
  \label{fig:octsplit}
\end{figure}
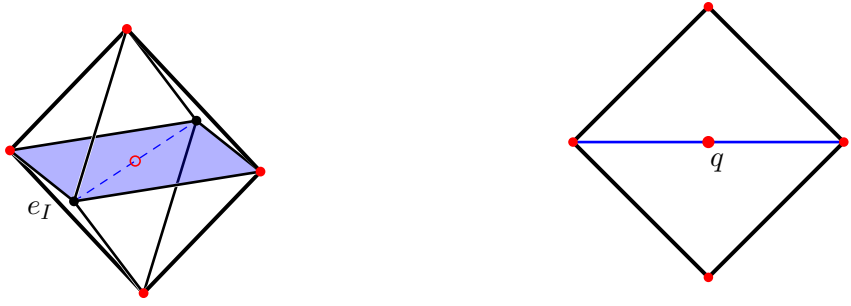

\begin{example}\label{ex:octsplit}
	A subdivision of the octahedron into two egyptian pyramids is a $2$-split.
   The common cell is a square.
   Figure~\ref{fig:octsplit} illustrates this subdivision as well as the induced subdivision of the vertex figure.
   The induced subdivision is a $2$-split of a square on a point configuration with five points, the four vertices and an interior point $q$.
   The point $q$ is the intersection of the vertex figure $\fig(e_I)$ and the convex hull of the two vertices that are not in the vertex figure. 
   The interior point $q$ is negligible.
\end{example}

The situation of Example~\ref{ex:octsplit} generalizes to $k$-splits of arbitrary polytopes.
\begin{proposition}\label{prop:induced_subdivision}
   Let $\Sigma$ be a $k$-split of the polytope $P$ and $v\in H^\Sigma$ be a vertex of $P$.
   Then each cone of $\cF$ intersects the vertex figure $\fig(v)$ of $v$.
   In particular, the subdivision $\Sigma$ induces a $k$-split on a point configuration that is contained in $\fig(v)$.
\end{proposition}
\begin{proof}
   Let us assume without loss of generality that the vertex $v$ is the apex of $\cF$.
   Each ray of $\cF$ is a cone of the form $\smallSetOf{v+\lambda (w-v)}{\lambda \geq 0}+\aff H^\Sigma$ for another vertex of $w\in P$.
   Hence, each ray intersects the vertex figure $\fig(v)$ of $v\in H^\Sigma$.
   This implies that the intersection of a $\ell$-dimensional cone with $\fig(v)$ is $\ell-1$ dimensional.
   We conclude that the induced subdivision is again a $k$-split.
\end{proof}

Our main goal is to classify all multi-splits of the hypersimplices.
Recall from Example~\ref{ex:vertex_fig} that for the hypersimplex the vertex figure of $e_I=\sum_{i\in I} e_i$ with $\size I = d$ is the product of simplices
\[
   \fig(e_I)\ =\ \SetOf{x\in\Delta(d,n)}{\sum_{i\in I} x_i = d-1} \ = \ \Delta(d-1,I)\times\Delta(1,[n]-I)\ \simeq \ \Delta_{d-1}\times\Delta_{n-d-1} \enspace.
\]
The intersection of the vertex figure of $e_I$ and the line spanned by the two vertices $e_I$ and $e_J$ with $J\in\tbinom{[n]}{d}$ is a point $q$ with coordinates
\[
   q_i = \left\{\begin{array}{cr}
      1                                & \text{ if } i \in I\cap J\\
      \frac{\size(I-J)-1}{\size(I-J)}  & \text{ if } i \in I-J\\
      \frac{1}{\size(I-J)}             & \text{ if } i \in J-I\\
      0                                & \text{ if } i \not \in I\cup J
         \end{array}\right.
\]
We denote by $\cQ_I$ the set of all these intersection points. They include the vertices of the vertex figure of $e_I$. For those we have $\size(I-J)=1$. 
A lifting function $\lambda$ of $\Delta(d,n)$ induces a lifting on each point $q\in\cQ_I$ by taking 
\[
   \lambda(q) = \lambda\left(\frac{\size(I-J)-1}{\size(I-J)} e_I + \frac{1}{\size(I-J)}  e_J\right) = \frac{\size(I-J)-1}{\size(I-J)} \lambda(e_I) + \frac{1}{\size(I-J)}  \lambda(e_J) \enspace .
\]

From Proposition~\ref{prop:induced_subdivision} follows that for each $k$-split of $\Delta(d,n)$ there exists a $d$-set $I$ and a vertex $e_I$ such that the $k$-split on $\Delta(d,n)$ induces a $k$-split on the point configuration $\cQ_I$.
Our goal is to show that all interior points of $\conv \cQ_I$ are negligible. 

Before we discuss this in general let us take a closer look on a key example where $d = n-d$.
In the example, the point configuration consists only of the vertices and exactly one additional point. 
This example will be central in the rest of the argumentation.

Consider the point configuration $\cP_j$ with the vertices of $\Delta_{j-1}\times\Delta_{j-1}$ and exactly one additional point $q$ which is $\sum_{i=1}^{2j} \frac{1}{j}e_i$ the barycenter of $\Delta_{j-1}\times\Delta_{j-1}$.

\begin{lemma}\label{lem:prod_simplices_one_int_point} There is no $(2j-1)$-split of $\cP_j$.
\end{lemma}
\begin{proof}
Let us assume we have given a $(2j-1)$-split $\Sigma$ of $\cP_j$.
The dimension of $\Delta_{j-1}\times\Delta_{j-1}$ is $2j-2$, hence the common cell $H^\Sigma$ is $(2j-2)-(k-1)$ dimensional. In our situation the dimension is $0$.
The only $0$ dimensional cell in the interior is $\{q\} = H^\Sigma$.
Let $\cF$ be the complete fan that induces $\Sigma$. The apex of $\cF$ has to be $q$.
Proposition~\ref{prop:ksplitHer} shows that this fan has $k=2j-1$ rays. 
Each of these $2j-1$ rays intersects $\Delta_{j-1}\times\Delta_{j-1}$ in a point on the boundary.
An intersection point has to be an element of the point configuration and hence it is a vertex of $\Delta_{j-1}\times\Delta_{j-1}$.
The convex hull $Q$ of all $2j-1$ vertices that we obtain as an intersection of the boundary with a ray is a $2j-2$-dimensional simplex in $\RR^{2j}$. 
This simplex $Q$ contains $q$ in its interior, since $\cF$ is complete. 
   By Lemma~\ref{lem:propertiesH} we have that $q$ is in the relative interior of $\conv\cP_j$. 
   Hence, no coordinate of $q$ is integral, while the vertices are $0/1$-vectors.
   This implies that for each of the $2j$ coordinates of $q$ there is a vertex of the simplex that is $1$ in this coordinate.
   A vertex of $\Delta_{j-1}\times\Delta_{j-1}$ has only two non zero entries, hence there is at least one coordinate $\ell\in[2j]$ such that only one vertex $w\in Q$ fulfills $w_\ell = 1$.
We deduce that the coefficient of $w$ in the  convex combination of the vertices that sums up to $q$ is $\frac{1}{j}$.

The simplex $Q$ is of dimension $2j-2$, which is the dimension of $\Delta_{j-1}\times\Delta_{j-1}$.
   Hence another vertex $v$ exists in $Q$, such that the support of $v$ and the support $w$ intersect non-trivially.
   The coefficient of $w$ is $\frac{1}{j}$, hence the coefficient of $v$ has to be $0$.
   This contradicts the fact that $q$ is in the interior of the simplex.
\end{proof}

\begin{remark}
   The proof of Lemma~\ref{lem:prod_simplices_one_int_point} shows that the barycenter of $\Delta_{j-1}\times\Delta_{j-1}$ is on the boundary of the constructed simplex $Q$.
   In fact, the arguments of the proof apply to any $(j+1)$-dimensional subpolytope of $\Delta_{j-1}\times\Delta_{j-1}$, instead of the subpolytope $Q$. 
   Hence, in any triangulation the barycenter is contained in a $j$-dimensional simplex.
\end{remark}

Our next step is to reduce the general case to the case where $2d = n$, which is equivalent to $\size I = d = n-d$, and the point configuration is $\cQ_I$. This is close to the situation in Lemma~\ref{lem:prod_simplices_one_int_point}, but still not the same.

For any non-vertex $p\in\cQ_I$ we define 
\[
   F_p=\SetOf{x\in\Delta(d,n)}{ \sum_{i\in I} x_i =d-1 \text{ and } x_j=p_j \text{ for all } p_j\in\{0,1\}  } \enspace.
\]
By definition the only point in $\cQ_I\cap \relint \fig(e_I)$ is $q$ and there is a unique $d$-set $J$ such that $q \in \conv(e_I,e_J)$.
Clearly $\size(I-J) = \size I - \size(I\cap J) = \size J - \size(I\cap J) = \size(J-I)$ and
\[
   q_j \text{ is non-integral if and only if } j\in I-J \text{ or } j\in J-I \enspace .
\]
The coordinatewise affine transformation $x_j\mapsto 1-x_j$ if $j\in I-J$ and $x_j\mapsto x_j$ if $j\in J-I$ is an isomorphism between the face $F_q$ of the vertex figure $\fig(e_I)$ and the product of simplices $\Delta_{j-1}\times\Delta_{j-1}$ for $j= d - \size I\cap J$. The point $q$ is mapped to the barycenter.
  
The common cell $H^\Sigma$ is either $\{q\}$ or $\fig(e_I)$.
Hence, the only possibilities for a multi-split of the point configuration $\cQ_I \cap F_q$ are $2j-1$ or $1$ maximal cell.
The multi-split is induced by the polytope $\Delta(d,n)\cap\aff\{e_I, F_q\}$.
A $1$-split can not be induced by a polytope. Therefore it has to be a $2j-1$-split. 
All together we get the following result for arbitrary multi-splits.
\begin{lemma}\label{lem:negible_points} 
   Let $\Sigma$ be a multi-split of the point configuration $\cQ_I$.
   All points of $\cQ_I\setminus\{0,1\}^n$ are negligible in $\Sigma$.
\end{lemma}
\begin{proof}
   To each $q\in\cQ_I$ we assign the set $\SetOf{i\in[n]}{q_i\not\in\ZZ}$ of non-integral support.
   A point $q\in\cQ_I$ is a $0/1$-vector if and only if its non-integral support is empty.
   Consider a ray $R$ in the fan $\cF$, i.e., the dimension of $R$ is $\dim(H^\Sigma)+1$.
   Let $V_R \subseteq \cQ_I$ be the set of all points of the intersection $R \cap \conv\cQ_I$.
   Fix a point 
   \[
      p \in \SetOf{ q\in V_R}{ \text{The non-integral support of $q$ is non empty} }
   \]
   whose non-integral support is inclusionwise minimal in the above set.
   Our goal is to show that such a $p$ does not exist and hence the above set is empty.
   This implies that any point $q\in V_R$ is integral.

   From \cite[Proposition 4.8]{Herrmann:2011} follows that the face $F_p$ is either trivially subdivided or a multi-split.
   In a trivial subdivision the interior point $p$ is not a vertex of $R\cap \conv\cQ_I$.
   By construction all the non integral points in $F_p$ except for $p$ are negligible, otherwise $p$ would not be a vertex of $V_R$.
   Moreover, $p$ is the only interior point and $k=2j-1$, where $j$ is the size of the non-integral support.
   This contradicts Lemma~\ref{lem:prod_simplices_one_int_point}.
   We conclude that the above constructed set is empty.
   Hence all non integral points in $\cQ_I$ are negligible.
\end{proof}

Proposition~\ref{prop:induced_subdivision} and Lemma~\ref{lem:negible_points} show that the induced subdivision is a subdivision of the vertex figure $\fig(e_I)$, which is a product of simplices.
This reverses a construction that lifts regular subdivisions of the product of simplices $\Delta_{d-1}\times\Delta_{n-d-1}$ to the hypersimplex $\Delta(d,n)$.
This lift has been studied in the context of tropical convexity in \cite{HerrmannJoswigSpeyer:2014}, \cite{Rincon:2013} and \cite{FinkRincon:2015}.
 We define the tropical Stiefel map of a regular subdivision on the product of simplices $\Delta_{d-1}\times\Delta_{n-d-1}$.
 We denote by $\lambda(i,j)\in\RR$ the height of the vertex $(e_i,e_j)\in\Delta_{d-1}\times\Delta_{n-d-1}$.
 The \emph{tropical Stiefel map} $\pi$ is defined on sets $A \subseteq \{1,\ldots,d\}, B\subseteq\{d+1,\ldots,n\}$ with $\size A = \size B$
 \[
    \pi: (A,B) \mapsto \min_{\omega \in \Sym(B)} \sum_{i\in A} \lambda(i,\omega_i)
 \]
 where $\Sym(B)$ is the symmetry group on the set $B$.
 Note that $\pi( \{i\},\{j\} ) = \lambda(i,j)$.

Let $e_I\in\Delta(d,n)$ be a vertex and $\lambda$ be a lifting on $\Delta_{d-1}\times\Delta_{n-d-1}$.
Then the tropical Stiefel map defines a lifting on a vertex $e_J\in\Delta(d,n)$ by taking the height $\pi(I-J, J-I)$.
The polytope $\Delta_{d-1}\times\Delta_{n-d-1}$ is isomorphic to $\fig(e_I) = \Delta_{d-1,I}\times\Delta_{1,[n]-I}\subsetneq\Delta(d,n)$. The tropical Stiefel map extends a lifting of the vertex figure $\fig(e_I)$ to the entire hypersimplex $\Delta(d,n)$.
The dual complex of the extended subdivision of $\Delta(d,n)$ is isomorphic to the dual complex of the subdivision of $\Delta_{d-1}\times\Delta_{n-d-1}$; see \cite[Theorem 7]{HerrmannJoswigSpeyer:2014}. In particular, the Stiefel map extends a $k$-split of $\Delta_{d-1}\times\Delta_{n-d-1}$ to a $k$-split of $\Delta(d,n)$.

  From Lemma~\ref{lem:negible_points} we deduce.
\begin{theorem} \label{thm:main1}
   Any $k$-split of the hypersimplex $\Delta(d,n)$ is the image of a $k$-split of a product of simplices $\Delta_{d-1}\times\Delta_{n-d-1}$ under the Stiefel map.
   In particular, the $k$-split $\Sigma$ is an extension of a $k$-split of $\Delta(d,I)\times\Delta(n-d,[n]-I)$ if and only if $e_I\in H^\Sigma$. 
\end{theorem}
\begin{proof}
   For any $k$-split $\Sigma$ of the hypersimplex $\Delta(d,n)$ and any vertex $e_I\in H^\Sigma$ the $k$-split $\Sigma$ induces a $k$-split on the point configuration $\cQ_I$.
   By Proposition~\ref{prop:negible} and Lemma~\ref{lem:negible_points} this is a subdivision on the vertex figure $\fig(e_I)$, which is a product of simplices.
   The Stiefel map extends this $k$-split to a $k$-split on $\Delta(d,n)$ by coning over the cells.
   This $k$-split coincides with $\Sigma$ on $\fig(e_I)$ and hence do both $k$-splits on the hypersimplex $\Delta(d,n)$.
\end{proof}

\section{Matroid subdivisions and multi-splits}\label{sec:nested}
\noindent
In this section we will further analyze multi-splits of the hypersimplex.
Our goal is to describe the polytopes that occur as maximal cells.
We will see that these polytopes correspond to a particular class of matroids.

A subpolytope $P$ of the hypersimplex $\Delta(d,n)$ is called a \emph{matroid polytope} if the vertex-edge graph of $P$ is a subgraph of the vertex-edge graph of $\Delta(d,n)$.
Note that the vertices of a matroid polytope are $0/1$-vectors and a subset of those of the hypersimplex.

The vertices of a matroid polytope $P$ are the characteristic vectors of the bases of a matroid $\matroid(P)$.
The convex hull of the characteristic vectors of the bases of a matroid $M$ is the matroid polytope $\polytope(M)$.
See \cite{Oxley:2011} and \cite{White:1986} for the basic background of matroid theory and \cite{Edmonds:1970} for a polytopal description, that we used as definition.

We will give three examples of classes of matroids that are important for this section.
\begin{example} 
   Clearly the hypersimplex $\Delta(d,n)$ itself is a matroid polytope.
   The matroid $\matroid(\Delta(d,n))$ is called \emph{uniform matroid} of rank $d$ on $[n]$ elements.
The $d$-subsets of $[n]$ are exactly the bases of $\matroid(\Delta(d,n))$.
   The uniform matroid has the maximal number of bases among all $(d,n)$-matroids.
\end{example}

\begin{example}\label{ex:partition} 
   Let $C_1,\ldots,C_k$ be a partition of the set $[n]$ and $d_i\leq\size(C_i)$ non-negative integers.
   The matroid $\matroid( \Delta(d_1,C_1)\times\cdots\times(\Delta(d_k,C_k) )$ is called \emph{partition matroid} of rank $d_1+ \ldots +d_k$ on $[n]$.
   A $d$-subset $S$ of $[n]$ is a basis of this matroid if $\size(S\cap C_i)=d_i$.
\end{example}

\begin{example} \label{ex:nested} 
   Let $\emptyset\subsetneq F_1\subsetneq \ldots \subsetneq F_k\subseteq[n]$ be an ascending chain of sets and $0 \leq r_1 < r_2 < \ldots < r_k$ be integers with $r_\ell < \size(F_\ell)$ for all $\ell\leq k$.
   The polytope
   \[
      P \ = \ \SetOf{x\in\Delta(d,n)}{\sum_{i\in F_\ell} x_i \leq r_\ell}
   \]
   is a matroid polytope. This follows from the analysis of all $3$-dimensional octahedral faces of the hypersimplex. Non of those is separated by more than one of the additional inequalities and hence the polytope is a matroid polytope. The matroid $\matroid(P)$ is called \emph{nested matroid} of rank $r_k+\size([n]-F_k)$ on $[n]$.
   The sets $F_1,\ldots,F_k$ are the \emph{cyclic flats} of the nested matroid $\matroid(P)$ if $r_1=0$.
   If $r_1\neq 0$, then the above and $\emptyset$ are the cyclic flats.
\end{example}

\begin{remark}
   There are many cryptomorphic definitions for matroids.
   Bonin and de Mier introduced in \cite{BoninMier:2008} the definition via cyclic flats and their ranks, i.e., unions of minimal dependent sets.
   In this paper we only need the very special case of nested matroids, where the lattice of cyclic flats is a chain.
\end{remark}

A \emph{matroid subdivision} of $\Delta(d,n)$ is a subdivision into matroid polytopes, i.e., all the (maximal) cells in the subdivision are matroid polytopes.
The lifting function of a regular subdivision of a matroid polytope is called \emph{tropical Pl\"ucker vector}, 
since it arises as valuation of classical Pl\"ucker vectors.
Note that the tropical Pl\"ucker vectors form a subfan in the secondary fan of the hypersimplex $\Delta(d,n)$. This fan is called the \emph{Dressian $\Dr(d,n)$}.

Each multi-split of the hypersimplex $\Delta(d,n)$ is a matroid subdivision as Theorem~\ref{thm:main1} in combination with the following proposition shows.
\begin{proposition}[{\cite{Rincon:2013},\cite{HerrmannJoswigSpeyer:2014}}]
   The image of any lifting function on $\Delta_{d-1}\times\Delta_{n-d-1}$ under the Stiefel map is a 
   tropical Pl\"ucker vector.
\end{proposition}

From now on let $\Sigma$ be a $k$-split of the hypersimplex $\Delta(d,n)$.
We investigate which matroid polytopes appear in the subdivision $\Sigma$.

Let us shortly introduce some matroid terms.
A set $S$ is \emph{independent} in the matroid $M$ if it is a subset of a basis of $M$.
The rank $\rank(S)$ of a set $S$ is the maximal size of an independent set in $S$. 
An important operation on a matroid $M$ is the \emph{restriction} $M|F$ to a subset $F$ of the ground set.
The set $F$ is the ground set of $M|F$.
A set $S\subseteq F$ is independent in $M|F$ if and only if $S$ is independent in $M$.
A matroid $M$ is called \emph{connected} if there is no set $\emptyset\subsetneq S \subsetneq [n]$ with
$\polytope(M) = \polytope(M|S)\times \polytope(M|([n]-S))$.
For each matroid $M$ there is a unique partition $C_1,\ldots,C_k$ of $[n]$, such that 
$\polytope(M) = \polytope(M|C_1)\times\cdots\times\polytope(M|C_k)$. The sets  $C_1,\ldots,C_k$ are called \emph{connected components} of $M$.
The element $e\in[n]$ is called \emph{loop} if $\{e\}$ is a connected component and $\polytope(M|\{e\}) = \Delta(0,\{e\})$.
If instead $\polytope(M|\{e\}) = \Delta(1,\{e\})$, then $e$ is called \emph{coloop}.
The dual operation of the restriction is the \emph{contraction} $M/F$.
The ground set of the matroid $M/F$ is $[n]-F$.
A set $S$ is independent in $M/F$ if $\rank_M(S+F)=\size(S)+\rank_M(F)$. 

The following describes a relation of the connected components of a matroid and its matroid polytope.
\begin{lemma}[{\cite[Theorem 3.2]{Fujishige:1984} and \cite[Propositions 2.4]{FeichtnerSturmfels:2005}}]\label{lem:dim}
   The number of connected components of a matroid $M$ on the ground set $[n]$
   equals the difference $n - \dim\polytope(M)$.
\end{lemma}

\begin{example}
   An element $e$ is a loop in a partition matroid $\matroid( \Delta(d_1,C_1)\times\cdots\times(\Delta(d_k,C_k) )$  if and only if $e\in C_\ell$ and $\rank(C_\ell) = d_\ell = 0$.
   The element is a coloop if instead $\rank(C_\ell) = d_\ell = \size(C_\ell)$.
   The other connected components are those sets $C_\ell$ with $0 < d_\ell < \size(C_\ell)$.

   A nested matroid is loop-free if $d_1 > 0$ and coloop-free if $F_k=[n]$.
   A  loop- and coloop-free nested matroid is connected.
\end{example}

At first we consider the common cell $H^\Sigma$ in a $k$-split $\Sigma$ of $\Delta(d,n)$.
\begin{proposition}
   The common cell $H^\Sigma$ is a matroid polytope of a loop and coloop-free partition matroid with $k$ connected components.
\end{proposition}
\begin{proof}
   The common cell $H^\Sigma$ is a cell in a matroid subdivision and hence a matroid polytope.
   The dimension of this polytope is $n-k+1$.
   From Lemma~\ref{lem:dim} follows that the corresponding matroid $M = \matroid(H^\Sigma)$ has $k$ connected components.
   Let $C_1,\ldots,C_k$ be the connected components of $M$ and $d_\ell=\rank_M(C_\ell)$.
   Clearly, this is a partition of the ground set $[n]$ and the sum $d_1+\ldots+d_k$ equals $d$.
   The polytope $H^\Sigma=\polytope(M)$ is the intersection of $\Delta(d,n)$ with an affine space.
   Hence, there are no further restrictions to the polytope and each matroid polytope $\polytope(M|C_\ell)$ is equal to $\Delta(d_\ell,C_\ell)$.
   The common cell $H^\Sigma$ intersects $\Delta(d,n)$ in the interior, hence $0 < d_\ell < \size(C_\ell)$ and the matroid $M$ is loop and coloop-free.
\end{proof}

We define the relation $\relation$ on the connected components $C_1,\ldots,C_k$ of $\matroid(H^\Sigma)$ depending on a cell $P\in\Sigma$ by
\begin{equation}\label{eq:relation}
\begin{aligned}
   C_a \relation C_b &\text{ if and only if for each $v\in H^\Sigma$ and for each $i\in C_a$ and $j\in C_b$ with}\\&\text{$v_i=1$ and $v_j=0$  we have } v+e_j-e_i\in P \enspace .
\end{aligned}
\end{equation}

\begin{lemma}
   Let $C_1,\ldots,C_k$ be the connected components of the matroid $\matroid(H^\Sigma)$.
   The matroid polytope $P$ of a cell in $\Sigma$ defines a partial order on the connected components $C_1,\ldots,C_k$ via $C_a \relation C_b$.
\end{lemma}
\begin{proof} Let $i,j\in[n]$ and $v\in H^\Sigma$ be a vertex with $v_i=1$, $v_j=0$.
   Then $v-e_i+e_j\in H^\Sigma$ if and only if there is a circuit in $\matroid(H^\Sigma)$ containing both $i$ and $j$.
 The vector $v$ is the characteristic vector of a basis in $\matroid(H^\Sigma)$ and adding $e_j-e_i$ corresponds to a basis exchange. This implies that $i$ and $j$ are in the same connected component, i.e., $\relation$ is reflexiv.

   Let $C_1 \relation C_2 \relation C_1$ and $i\in C_1$, $j \in C_2$.
   Take $v,w\in H^\Sigma$ with $v_i=w_j=1$ and $v_j=w_i=0$.
   By assumption we have $v-e_i+e_j, w+e_i-e_j\in P$ and since $H^\Sigma$ is convex
   \[
      \frac{1}{2}(v-e_i+e_j) +  \frac{1}{2}(w+e_i-e_j) = \frac{1}{2} v + \frac{1}{2} w \in H^\Sigma \enspace .
   \]
   A convex combination of points in $P$ lies in $H^\Sigma$ if and only if all the points are in $H^\Sigma$.
   Hence, we got $v-e_i+e_j\in H^\Sigma$ and therefore $C_1 = C_2$.

   Let $C_1 \relation C_2 \relation C_3$, $i\in C_1$,  $j\in C_3$ and $v\in H^\Sigma$ with $v_i=1$ and $v_j=0$.
   Consider the cone $Q = \smallSetOf{\lambda x + y }{y\in H^\Sigma, x+y \in P \text{ and } \lambda \geq 0  }$.
   This is the cone in the fan $\cF$ that contains $P$ with the same dimension as $P$.
   Let $k,\ell\in C_2$ be indices with $v_k = 1$ and $v_\ell = 0$ and $w=v-e_k+e_\ell$.
   Then $v,w\in H^\Sigma$ and $v-e_k+e_j,\, w+e_k-e_\ell,\, v-e_i+e_\ell\,\in P$.
   That implies
   \[
      v-e_i+e_j \ =\ \frac{1}{3}\left( v+3(e_j-e_k) + w+4(e_k-e_\ell)+ v+3(e_\ell-e_i)\right) \in Q \enspace .
   \]
   Clearly $v-e_i+e_j\in \Delta(d,n)$ and hence $v-e_i+e_j\in P$.
   This shows that $\relation$ is transitive.
\end{proof}
   
Before we further investigate the relation $\relation$ we take a look at rays of $\cF$.
The next Lemma describes the $(n-k+2)$-dimensional cells in $\Sigma$.
The $k$-split $\Sigma$ has exactly $k$ of these cells and each maximal cell contains $k-1$ of those; see Proposition~\ref{prop:ksplitHer}.
\begin{lemma}\label{lem:rays}
   For each $(n-k+2)$-dimensional cell of $\Sigma$ there are $a,b\in[n]$ such that the cell equals
   \[
      R_{a,b}\ = \ \left( H^\Sigma + \SetOf{\mu(e_a-e_b)}{\mu\geq 0} \right)\cap \Delta(d,n)  \enspace .
   \]
\end{lemma}
\begin{proof}
   Let $v$ be a vertex of the $(n-k+2)$-dimensional cell $R$ that is not in $H^\Sigma$.
   This cell $R$ is a matroid polytope.
   Hence, there is an edge of $v$ that has the direction $e_i-e_j$ for some $i$ and $j$.
   At least one of those edges connects $v$ with $H^\Sigma$.
   Therefore $R$ is  of the desired form.
\end{proof}

Now we are able to further investigate $\relation$ and hence the cells in the $k$-split $\Sigma$.
\begin{lemma}
   For a connected matroid $\matroid(P)$ the relation $\relation$ is a total ordering on the connected components of $\matroid(H^\Sigma)$.
\end{lemma}
\begin{proof}
   Let us assume that $C_1$ and $C_2$ are two incomparable connected components of $M(H^\Sigma)$.
   We define
   \[
      F = \bigcup_{C \relation C_1} C \text{ and } G = \bigcup_{C \relation C_2} C \enspace .
   \]
   Pick $i\in C_1$ and $j\in C_2$. The matroid $M = \matroid(P)$ is connected hence there is a circuit $A$ containing both $i$ and $j$. 
   The set $A\cap F\cap G$ is independent in $M$, as $i\not\in G$.
   Let $S\supseteq A\cap F\cap G$ be a maximal independent set in $F\cap G$.
   Let $N$ be the connected component of $i$ in the minor $(M/S)|([n]-F\cap G)$.
   Note that the elements of $F\cap G-S$ are exactly the loops in the contraction $M/S$.
   Moreover, $A-S$ is a circuit in $M/S$, and hence is $j\in N$.
   We conclude that $C_1,C_2\subset N$.

   The equation $\sum_{\ell\in N} x_\ell = \rank(N)$ defines a face of $\polytope(M)$.
   This face is contained in $\polytope(N)\times\Delta(d-\rank(N),[n]-N)\subsetneq\Delta(\rank(N),N)\times\Delta(d-\rank(N),[n]-N)$.
   
   \cite[Proposition 4.8]{Herrmann:2011} states that the induced subdivision on a face of a $k$-split is either trivial or a multi-split with less than $k$ maximal cells.
   We are in the latter case, as the induced subdivision on $\Delta(\rank(N),N)$ is not trivial, since $C_1$ and $C_2$ are contained in $N$.
   
   Hence, we can assume without loss of generality that $F\cap G = \emptyset$.
   Clearly, the following two inequalities are valid for $\polytope(M)$ and the face that they define includes $H^\Sigma$
   \begin{align*}
      \sum_{i\in F} x_i \leq \rank(F) \text{ and } \sum_{i\in G} x_i \leq \rank(G) \enspace .
   \end{align*}
   Let $R$ be the unique ray in $\Sigma$ that is not contained in $\polytope(M)$.
   There is a vertex $v\not\in H^\Sigma$ of $\Delta(d,n)$ that is contained in both $R$ and in $H^\Sigma-e_a+e_b$ for some $a,b\in[n]$. The rays in $\Sigma$ positively span the complete space. Hence, we get the estimation
   \begin{align*}
      \rank(F)+1 \geq \sum_{i\in F} v_i > \rank(F) \text{ and } \rank(G)+1 \geq \sum_{i\in G} v_i > \rank(G) \enspace .
   \end{align*}
   This implies that $b\in F\cap G$. We conclude that either $F \relation G$ or  $G \relation F$.
\end{proof}

\begin{example}\label{ex:2split}
   Consider the octahedron $\Delta(2,4)$.
   The hyperplane $x_1+x_2=x_3+x_4$ through the four vertices $e_1+e_3$, $e_2+e_3$, $e_1+e_4$ and $e_2+e_4$   
   strictly separates the vertices $e_1+e_2$ and $e_3+e_4$.
   Moreover the hyperplane splits $\Delta(2,4)$ into two maximal cells, the corresponding subdivision $\Sigma$ is a $2$-split. The partition matroid $\matroid(H^\Sigma)$ has four bases and two connected components $C_1 = \{1,2\}$ and $C_2 = \{3,4\}$.

   Let $M$ be the $(2,4)$-matroid with the following five bases $\{1,3\},\{1,4\},\{2,3\},\{2,4\},\{3,4\}$.
   The polytope $\polytope(M)$ is an egyptian pyramid and a maximal cell in $\Sigma$.
   The inequality $x_1+x_2\leq 1$ is valid for $\polytope(M)$ and hence $C_2\not\relation C_1$.
   It is easy to verify that $C_1\relation C_2$ as $e_3+e_4\in \polytope(M)$.
\end{example}

We derive the following description for the maximal cells of a $k$-split, which we already saw in Example~\ref{ex:2split}.
\begin{lemma} \label{lem:singlepoltope}
   Let $P$ be a maximal cell of the $k$-split $\Sigma$ of $\Delta(d,n)$.
   Furthermore, let $C_1\relation\ldots\relation C_k$ be the order of the connected components of $M(H^\Sigma)$.
   Then $x\in P\subsetneq \RR^n$ if and only if $x\in\Delta(d,n)$ with
   \begin{align} \label{eq:singlepoltope}
      \sum_{\ell=1}^h \sum_{i\in C_\ell} x_i &\leq \sum_{\ell=1}^h \rank_M(C_\ell) \text{ for } h\leq k \enspace .
   \end{align}
\end{lemma}
\begin{proof}
   First, we will show that each $x\in P$ fulfills the inequalities (\ref{eq:singlepoltope}).
   The following equation holds for each $v\in H^\Sigma\subsetneq P$ 
   \[
      \sum_{i\in C_\ell} v_i \ =\  \rank_M(C_\ell) \enspace .
   \]
   Lemma~\ref{lem:rays} shows that a ray of $\cF$ is of the form $H^\Sigma + \pos(e_j-e_i)$ for some $i,j\in [n]$.
   Clearly, for each pair $(i,j)$ of such elements and every point 
   $v\in H^\Sigma$ with coordinates $v_j=0$ and $v_i=1$ we get $v+e_j-e_i\in \Delta(d,n)-H^\Sigma$.

   Hence, $v+e_j-e_i\in P$ implies that $C_a\relation C_b$ for $i\in C_a$ and $j\in C_b$. This is $a\leq b$.
   This proves (\ref{eq:singlepoltope}) for all points that are in a ray and in $P$.
   Each point $x\in P$ is a positive combination of vectors in rays of the fan $\cF$, hence the inequalities(\ref{eq:singlepoltope}) are valid for all vectors in $P$.

   Conversely, we will show that each point in $\Delta(d,n)$, that is valid for (\ref{eq:singlepoltope}), is already in $P$.
   The left hand side of (\ref{eq:singlepoltope}) is a totally unimodular system, i.e., all square minors are either $-1$, $0$ or $1$.
   Hence all the vertices of the polytope are integral, even if we add the constraints $0 \leq x_i \leq 1$.
   This is precisely a statement of \cite[Theorem 19.3]{Schrijver:1986}.

   Take a vertex $v$ of $\Delta(d,n)$ that is valid.
   Either $v\in H^\Sigma$ and hence $v\in P$ or at least an inequality of (\ref{eq:singlepoltope}) is strict.
   In this case let $a = \min\smallSetOf{\ell\in[n]}{ \sum_{i\in C_\ell} v_i <  \rank_M(C_\ell) }$ and  $b = \min\smallSetOf{\ell\in[n]}{ \sum_{i\in C_\ell} v_i >  \rank_M(C_\ell) }$.
   Note that both sides of the inequality (\ref{eq:singlepoltope}) for $h=k$ sum up to $d$.
   Hence, both of the minima exist and $a<b$, otherwise the inequality (\ref{eq:singlepoltope}) with $h=a$ would be invalid.
   Pick $i\in C_b$ with $v_i=0$ and $j\in C_a$ with $v_j=1$.
   The vector $w = v-e_j+e_i$ is another vertex of $\Delta(d,n)$, that is valid for (\ref{eq:singlepoltope}).
   Moreover, $w\in P$ implies that $v\in P$ since $C_a\relation C_b$.
   We conclude that $P$ has the desired exterior description.
\end{proof}

Now we are able to state our second main result, which allows us to construct all $k$-splits of the hypersimplex explicitly and relate them to nested matroids.
\begin{theorem}\label{thm:main2}
   A maximal cell in any $k$-split $\Sigma$ of $\Delta(d,n)$ is the matroid polytope $P(M)$ of a connected nested matroid $M$.

   More precisely, the cyclic flats of $M$ are the $k+1$ sets $\emptyset\subsetneq C_1 \subsetneq C_1\cup C_2  \subsetneq \ldots \subsetneq \bigcup_{i=1}^k C_i=[n]$, where $C_1\relation \ldots \relation C_k$ are the connected components of $M(H^\Sigma)$.

   Moreover, the other $k$ maximal cells are given by a cyclic permutation of the sets $C_1,C_2\ldots,C_k$.
   In particular, each maximal cell in a multi-split of $\Delta(d,n)$ determinates all the cells.
\end{theorem}
\begin{proof}
   Fix a maximal cell $P$ in $\Sigma$ and let $C_1\relation\ldots\relation C_k$ be the connected components of the partition matroid $N=\matroid(H^\Sigma)$.
   We define $F_\ell=\bigcup_{i=1}^\ell C_i$ for all $1\leq \ell\leq k$.
   We have
   \[
      0 < \rank_N( F_1 ) < \ldots < \rank_N(F_{\ell-1}) < \rank_N(F_{\ell-1})+\rank_N(C_\ell)=\rank_N(F_\ell) < \ldots < \rank_N(F_k) = d \enspace .
   \]
   The sets $F_\ell$ and $\emptyset$ are the cyclic flats of nested matroid $M$ with ranks given by $\rank_N(F_\ell)$ resp. $0$; see Example~\ref{ex:nested}.   
   The matroid polytope $P(M)$ of $M$ is exactly described by Lemma~\ref{lem:singlepoltope}.
   This implies that the maximal cell $P$ is the matroid polytope $P(M)$ with the desired $k+1$ cyclic flats.

   The intersection of all maximal cells of the $k$-split $\Sigma$ excluded the cell $P(M)$ is a ray of $\cF$.
   This ray $R_{a,b}$ contains a vertex $w\in\Delta(d,n)$ of the form $v+e_a-e_b$, where $v\in H^\Sigma$.
   We can choose this vertex $w$, such that $w\not\in P(M)$. We deduce from \eqref{eq:singlepoltope} the following strict inequalities for $w$:
   \[
   \sum_{\ell=1}^h \sum_{i\in C_\ell} w_i \ >\ \sum_{\ell=1}^h \rank_M(C_\ell) \text{ for all } h< k \enspace .
   \]
   As $w=v+e_a-e_b$ and $\sum_{i\in C_\ell} v_i = \rank(C_\ell)$, we get for $h=1$ that $a\in C_1$ and from $h=k-1$ that $b\in C_k$.
   This implies that for every maximal cell $Q\neq \polytope(M)$ of $\Sigma$ we have $C_k \relationQ C_1$.

   Moreover, each maximal cell $Q\neq \polytope(M)$ shares a facet with $\polytope(M)$.
   Let $\sum_{\ell=1}^m \sum_{i\in C_\ell} x_i  = \sum_{\ell=1}^m \rank_M(C_\ell)$ be the facet defining equation.
   This facet implies $ C_m \not\relationQ C_{m+1}$.
   All the other inequalities of (\ref{eq:singlepoltope}) are valid for $Q$.
   We conclude that $ C_{m+1} \relationQ \ldots \relationQ C_k \relationQ C_1 \relationQ \ldots \relationQ C_m$.
\end{proof}
Note that there is a finer matroid subdivision for any $k$-split of the hypersimplex $\Delta(d,n)$, except for the case $k=d=2$ and $n=4$.
Moreover, each matroid polytope of a connected nested matroid with at least four cyclic flats on at least $k+d+1$ elements occurs in a coarsest matroid subdivision, which is not a $k$-split.

In contrast we have that for each connected nested $(d,n)$-matroid $M$ with $k+1$ cyclic flats there is a unique $k$-split of the hypersimplex $\Delta(d,n)$ that contains $\polytope(M)$ as a maximal cell.
Conversely, a $k$-split of the hypersimplex $\Delta(d,n)$ determines $k$ such nested matroids.
Furthermore, each $k$-split $\Sigma$ determines a unique loop- and coloop-free partition matroid $\matroid(H^\Sigma)$, while each ordering of the connected components of $\matroid(H^\Sigma)$ leads to a unique connected nested $(d,n)$-matroid with $k+1$ cyclic flats.
We conclude the following enumerative relations.
\begin{corollary}
   The following three sets are pairwise in bijection:
   \begin{enumerate}
      \item The loop- and coloop-free partition $(d,n)$-matroids with $k$ connected components, 
      \item\label{item:b} the collections of all connected nested $(d,n)$-matroids with $k+1$ cyclic flats,
         whose pairwise set differences of all of those cyclic flats coincide,
      \item\label{item:c} the collections of $k$-splits of $\Delta(d,n)$ with the same interior cell.
   \end{enumerate}
   Moreover, the collections in (\ref{item:b}) have all the same size $k!$
   and those in (\ref{item:c}) are of size $(k-1)!$.
\end{corollary}

Now we are able to count $k$-times all $k$-splits of the hypersimplex $\Delta(d,n)$ by simply counting nested matroids, i.e., ascending chains of subsets. The following is a natural generalization of the formulae that count $2$-splits in \cite[Theorem 5.3]{HerrmannJoswig:2008} and $3$-splits in \cite[Corollary 6.4]{Herrmann:2011}.
\begin{proposition}\label{prop:enum_ksplits}
 The total number of $k$-splits in the hypersimplex $\Delta(d,n)$ equals
   \[
      \frac{1}{k}\sum_{a_1=2}^{\beta_1 - 2(k-1)} \cdots \sum_{\alpha_{k-1}=2}^{\beta_{k-1}-2} \mu_k^{d,n}(\alpha_1,\ldots,\alpha_{k-1}) \prod_{j=1}^{k-1}\binom{\beta_j}{\alpha_j}
   \]
where $\beta_i = n-\sum_{\ell=1}^{i-1}\alpha_\ell$ and
   \[
      \mu_k^{d,n}(\alpha_1,\ldots,\alpha_{k-1}) \ =\ \size\left(\SetOf{x\in\ZZ^k}{\sum_{i=1}^k x_i=d
      \text{ and } 0 < x_j < \alpha_j \text{ for $j\leq k$}}\right)
   \]
with $\alpha_k=\beta_k$.
\end{proposition}
\begin{proof}
   Fix non-negative numbers $\alpha_1,\ldots,\alpha_k$ that sum up to $n$.
   The number of connected nested $(d,n)$-matroids with $k+1$ cyclic flats $\emptyset=F_0\subsetneq F_1,\ldots,F_k=[n]$ that satisfy $\size(F_j-F_{j-1})=\alpha_j$ is determinated by the following product of binomial coefficients weighted by the number $\mu^{d,n}_k$ of possibilities for ranks on the cyclic flats
   \[
    \mu_k^{d,n}(\alpha_1,\ldots,\alpha_{k-1}) \prod_{j=1}^k\binom{n-\alpha_1-\ldots-\alpha_{j-1}}{\alpha_j} \enspace .
   \]
Clearly, the rank function satisfies $0 < \rank( F_j )-\rank( F_{j-1} )   <  \size( F_j-F_{j-1} ) = \alpha_j$, hence $\alpha_j\geq 2$. Moreover, the last binomial coefficient is equal to one.
The number $\alpha_k$ is determinated by $\alpha_k=n-\sum_{j}^{k-1}\alpha_j$.
We get that the number of connected nested $(d,n)$-matroids with $k+1$ cyclic flats is given by
  \[
       \sum_{a_1=2}^{\beta_1 - 2(k-1)} \cdots \sum_{\alpha_{k-1}=2}^{\beta_{k-1}-2} \mu_k^{d,n}(\alpha_1,\ldots,\alpha_{k-1}) \prod_{j=1}^{k-1}\binom{\beta_j}{\alpha_j} \enspace .
   \]
   We derive the number of $k$-splits by division by $k$.  
   This completes the proof.
\end{proof}

\begin{example}
   Consider the case that $n=d+k=2k$.
   The number of loop- and coloop-free partition $(k,n)$-matroids equals $(2k-1)!! = (2k-1)(2k-3)\cdots 1$,
   as $\alpha_j=2$ for all $j\leq k$.
   The number of $k$-splits in $\Delta(k,2k)$ equals $(k-1)! (2k-1)!!$ and those of connected nested matroids with $k+1$ cyclic flats $k! (2k-1)!!$. 
   Note that in this case all these $k$-splits, partitions and nested matroids are equivalent under reordering of the $[n]$ elements.
\end{example}

Combining Theorem~\ref{thm:main1} and Theorem~\ref{thm:main2} leads to an enumeration of all $k$-splits of the product of simplices $\Delta_{d-1}\times\Delta_{\ell-1}$, by splitting the connected component $C_j$ into $A_j$ and $B_j$ with $\rank(C_j) = \size(A_j)$.
Note that the number of $k$-splits of a product of simplices can not simply be derived from the number of  $k$-splits of a hypersimplex by double counting, since each $k$-split is covered by multiple vertex figures whose number depends on the $k$-split and no product of simplices covers all $k$-splits of a hypersimplex.
\begin{theorem}\label{thm:simplices}
   The $k$-splits of $\Delta_{d-1}\times\Delta_{\ell-1}$ are in bijection with 
   collections of $k$ pairs $(A_1,B_1),\ldots,(A_k,B_k)$, such that
   $A_1,\ldots,A_k$ is a partition of $[d]$ and $B_1,\ldots,B_k$ is a partition of $[\ell]$.
   In particular, the number of $k$-splits of $\Delta_{d-1}\times\Delta_{\ell-1}$ equals
   \[
      \frac{1}{k} \left(
      \sum_{\alpha_1=1}^{\beta_1-(k-1)}  \cdots \sum_{\alpha_{k-1}=1}^{\beta_{k-1}-1} \, 
      \prod_{j=1}^{k-1}\binom{\beta_j}{\alpha_j} \right)\cdot \left(
      \sum_{\gamma_1=1}^{\delta_1-(k-1)} \cdots \sum_{\gamma_{k-1}=1}^{\delta_{k-1}-1} \,
      \prod_{j=1}^{k-1}\binom{\delta_j}{\gamma_j} \right) \enspace ,
   \]
where $\beta_i = d-\sum_{j=1}^{i-1} \alpha_j$ and $\delta_i = \ell-\sum_{j=1}^{i-1} \gamma_j$ .
\end{theorem}

\section{coarsest matroid subdivisions}\label{sec:computations}
\noindent
We have enumerated specific coarsest matroid subdivisions.
In this section we will compare two constructions for coarsest matroid subdivisions.
We have seen already the first of these constructions for matroid subdivisions.
The Stiefel map lifts rays of $\Delta_{d-1}\times\Delta_{n-d-1}$ to rays of the Dressian $\Dr(d,n)$.
This construction for rays has been studied in \cite{HerrmannJoswigSpeyer:2014} under the name of \enquote{tropically rigid point configurations}. 
Other (coarsest) matroid subdivisions can be constructed via matroids.
Let $M$ be a $(d,n)$-matroid. The \emph{corank vector} of $M$ is the map 
\[
   \rho_M: \binom{[n]}{d} \to \NN ,\quad S \mapsto d-\rank_M(S) \enspace .
\]
The corank vector is a tropical Pl\"ucker vector.
Moreover, the induced subdivision contains the matroid polytope $\polytope(M)$ as a cell; see \cite[Example 4.5.4]{Speyer:2005} and \cite[Proposition 34]{JoswigSchroeter:2017}.

There are coarsest matroid subdivisions, obtained from corank vectors, that are not in the image of the Stiefel map; see \cite[Figure 7]{HerrmannJoswigSpeyer:2014} and \cite[Theorem 41]{JoswigSchroeter:2017}.

There are matroid subdivisions that are both, induced by the Stiefel map and corank subdivisions.
\begin{example}
   We have seen in Theorem~\ref{thm:main1}, that every multi-split of the hypersimplex is induced by the Stiefel map.
   Moreover, each multi-split is a corank subdivision. The maximal cells are  nested matroids. This follows from Theorem~\ref{thm:main2} combined with the methods of \cite[Section 4]{JoswigSchroeter:2017}.
\end{example}

A subdivision that is induced by a corank vector satisfies the following criteria.
With these we are able to certify that a matroid subdivision is not induced by a corank vector.
\begin{lemma}\label{lem:corank}
   Let $M$ be a $(d,n)$-matroid and $\Sigma$ the corank subdivision of $\polytope(M)$.
   For each vertex $v$ of the hypersimplex $\Delta(d,n)$ a (maximal) cell $\sigma\in\Sigma$ exists, such that 
   $v\in\sigma$ and $\sigma\cap\polytope(M)\neq\emptyset$.
   In particular, the cell $P(M)$ together with the neighboring cells cover all vertices of $\Delta(d,n)$.
\end{lemma}
\begin{proof}
   Let $M$ be a $(d,n)$-matroid and $\Sigma$ the corresponding corank subdivision of the hypersimplex $\Delta(d,n)$.
   Furthermore, let $v$ be a vertex of the hypersimplex $\Delta(d,n)$.
   Then $v = e_S$ for a set $S\in\tbinom{[n]}{d}$.
   Given a basis of $M$ and a maximal independent subset of $S$, the set $S$ can be enlarged to a basis with $d-\rank(S)$ elements of the basis.
   Hence, there is a sequence $S_{d-\rank(S)},\ldots,S_0\in\tbinom{[n]}{d}$, 
   such that $S_{d-\rank(S)}=S$, $\size(S_j\cap S_{j+1})=d-1$ and $d-\rank(S_j) = j$ for all $0 \leq j< d-\rank(S)$.
   Thus, the corresponding $d-\rank(S)+1$ vertices of the lifted polytope of $\Delta(d,n)$ lie on the hyperplane $\sum_{i\in S_0}x_i=x_{n+1}$, where $x_{n+1}$ is the height coordinate, i.e. the corank. This hyperplane determinates a face of the lifted polytope and hence a cell $\sigma\in\Sigma$.
   Both vertices $v$ and $e_{S_0}\in\polytope(M)$ are contained in $\sigma$.
\end{proof}

\begin{lemma}\label{lem:pc_corank}
   Let $\Sigma$ be a subdivision of the hypersimplex $\Delta(d,n)$, such that
the subdivision is the corank subdivision of a connected matroid
 and induced by a regular subdivison of the product of simplices via the Stiefel map.
The subdivision on the product of simplices $\Delta_{d-1}\times\Delta(n-d-1)$ is 
realizable with a $0/1$-vector as lifting function.
\end{lemma}
\begin{proof}
   Clearly, the corank subdivision $\Sigma$ of the matroid $M$ is regular.
   Moreover, if $\Sigma$ is induced by the Stiefel map, then there is a vertex $v$ that is contained in each maximal cell.   
   The matroid polytope $\polytope(M)$ is a maximal cell as $M$ is connected.
   Hence, the vertex $v$ is a vertex of $\polytope(M)$ and the characteristic vector of a basis of $M$.
   This implies that the neighbours of $v$ are of corank $0$ and $1$.
   This shows that the restriction of the corank lifting to the neighbours of $v$ has the required form.
\end{proof}

We will apply Lemma~\ref{lem:pc_corank} to tropical point configurations. These are vectors in the tropical torus $\RR^d/(1,\ldots,1)\RR$. 
The line segment in the tropical torus between the two points $v$ and $w$ is the set $\smallSetOf{u\in \RR^d/(1,\ldots,1)\RR}{\lambda,\mu\in\RR \text{ and } u_i = \min (v_i+\lambda,\, w_i+\mu) }$.
Note that such a line segment consists of several ordinary line segments, with additional (pseudo-)vertices.
The \emph{tropical convex hull} of a set of points is the smallest set such that all line segments between points are in this set.
Such a tropical convex hull of finitely many points decomposes naturally in a polyhedral complex.
The cells in the tropical convex hull of a tropical point configuration of $(n-d)$ points in $\RR^d/(1,\ldots,1)\RR$ are in bijection with the cells of a regular subdivision of the product $\Delta_{d-1}\times\Delta_{n-d-1}$, where the height of $e_i+e_j$ is the $j$-th coordinate of the $i$-th point in the tropical point configuration; see \cite[Lemma 22]{DevelinSturmfels:2004}.
A tropical point configuration is \emph{tropically rigid} if it induces a coarsest (non-trivial) subdivision on the product of simplices $\Delta_{d-1}\times\Delta_{n-d-1}$.

A tropical point configuration corresponds to a corank subdivision if the points are realizable by $0/1$ coordinates in $\RR^d$ or equivalently by $-1$, $0$ and $1$ in the tropical torus.
In particular, there is a point that has lattice distance at most one to each other point.
This criteria certify that the next examples are not corank subdivisions.
\begin{figure}
   \input{troppointconf.tikz}
   \caption{The nine rigid tropical point configurations of Example~\ref{ex:nine_tpc}, each of which is a tropical convex hull of six points.}
  \label{fig:troppointconf}
\end{figure}
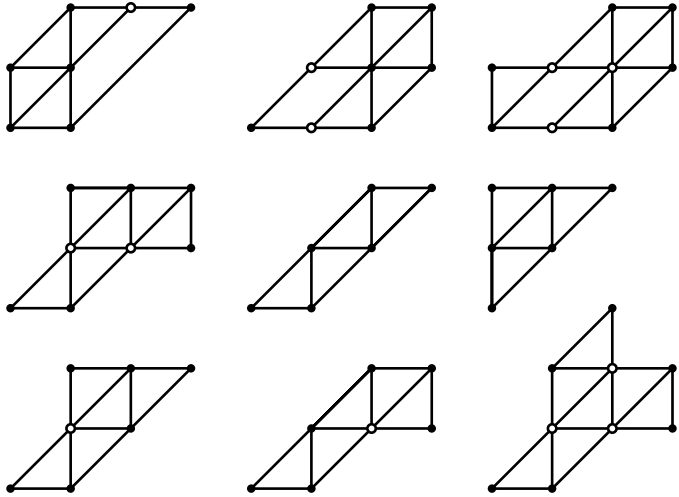

The following illustrates examples of coarsest non-corank subdivisions. 
\begin{example}\label{ex:nine_tpc}
   Figure~\ref{fig:troppointconf} shows nine rigid tropical point configurations out of $36$ symmetry classes.
   They correspond to nine coarsest subdivisions of $\Delta_2\times\Delta_5$.
   The Stiefel map of those induces coarsest matroid subdivisions of the hypersimplex $\Delta(3,9)$.
   None of those is a  corank subdivision.
   Proposition~\ref{prop:comp} shows that these are all rigid tropical point configurations that do not lift to a corank subdivision of the hypersimplex $\Delta(3,9)$.
\end{example}
We lifted those to rays of the hypersimplex $\Delta(3,9)$ and checked whether they are equivalent to corank liftings.
For this computation we used both the software \polymake \cite{DMV:polymake} and \mptopcom \cite{JordanJoswigKastner:2017}. 
Before we state our computational result, note that there is a natural symmetry action of the symmetric group on $n$ elements on the hypersimplex $\Delta(d,n)$.
This group acts on the hypersimplex, by permutation of the coordinate directions.
From our computations we got the following result.

\begin{proposition}\label{prop:comp} The nine liftings illustrated as tropical point configuration in Figure~\ref{fig:troppointconf} lead to coarsest regular subdivisions of $\Delta(3,9)$.
These are, up to symmetry, all coarsest regular subdivisions of $\Delta(3,9)$ that are induced by the Stiefel map and not by a corank lift.
\end{proposition}

We will close with two enumerative results about the number of coarsest regular matroid subdivisions of the hypersimplex $\Delta(d,n)$ for small parameters $d$ and $n$.
With the previously mentioned methods we have computed all coarsest regular subdivisions of $\Delta_{d-1}\times\Delta_{n-d-1}$ for small parameters of $d$ and $n$ and lifted them to the hypersimplex. 
Note that this is a massive computation, as there are $7402421$ symmetry classes of triangulations for the product $\Delta_{3}\times\Delta_{4}$ and the acting symmetric group has $9!$ elements.
Another example is $\Delta_{2}\times\Delta_{6}$ where the number of symmetry classes of triangulations in the regular flip component is $533242$ and the group has $10!=3628800$ elements.
For each symmetry class a convex hull computation is necessary and after that another reduction that checks for symmetry.

  The number of all these subdivisions up to symmetry is listed in Table~\ref{tab:stiefel} on the last page.
Note that we do not count the number of coarsest regular subdivisions of $\Delta_{d-1}\times\Delta_{n-d-1}$.

For our second result we computed all corank subdivisions for all matroids in the \polymake database available at \href{https://db.polymake.org}{db.polymake.org} . This database is based on a classification of matroids of small rank with few elements of Matsumoto, Moriyama, Imai and Bremner \cite{Matsumotoetal:2012}.
We got the coarsest subdivisions by computing the secondary cones.
The number of all of these subdivisions is given in Table~\ref{tab:corank}. 

Combining both techniques we got the following result.
\begin{proposition}
   The number of coarsest matroid subdivisions of $\Delta(d,n)$ for $d\leq 4$ and $n\leq 10$, excluded $d=4$, $n=10$, is bounded from below by the numbers listed in Table~\ref{tab:total}.
\end{proposition}

\begin{table}[t]
  \centering
   \caption{Numbers of symmetry classes of coarsest matroid subdivisions in the hypersimplex $\Delta(d,n)$.}
  \begin{subtable}[c]{0.38\textwidth}
     \subcaption{The number in the Stiefel image.\label{tab:stiefel}}
     \begin{tabular}{l@{\hskip 2.5mm}rrrrrrr}
        \toprule
        $d\backslash n$ & 4 & 5 & 6 & 7 & 8 & 9 & 10\\
        \midrule
        2&   1 & 1& 2& 2&  3&  3&  4\\
        3&     & 1& 3& 5& 11& 36& 207\\
        4&     &  & 2& 5& 39& 2949& --\\
        \bottomrule
     \end{tabular}

  \end{subtable}
   \hspace{1.5cm}
  \begin{subtable}[c]{0.38\textwidth}
     \subcaption{The number of corank subdivisions.\label{tab:corank}}
     \begin{tabular}{l@{\hskip 2.5mm}rrrrrrr}
        \toprule
        $d\backslash n$ & 4 & 5 & 6 & 7 & 8 & 9 & 10\\
        \midrule
        2&     1 & 1& 2& 2&  3&   3&   4\\
        3&       & 1& 3& 5& 12&  38& 139\\
        4&       &  & 2& 5& 33& 356&  --\\
        \bottomrule
     \end{tabular}
  \end{subtable}
\end{table}

\begin{table}[t]
  \centering
   \caption{The number of coarsest matroid subdivisions in $\Delta(d,n)$ that are either corank subdivisions or in the image of the Stiefel map.}
  \label{tab:total}
  \begin{subtable}[c]{0.58\textwidth}
     \subcaption{The number without any identifications.}
     \begin{tabular}{l@{\hskip 2.5mm}rrrrrrr}
        \toprule
        $d\backslash n$ & 4 & 5 & 6 & 7 & 8 & 9 & 10\\
        \midrule
         2&     3& 10& 25&   56&    119&      246&      501\\
         3&      & 10& 65&  616&  15470&  1220822& 167763972\\
         4&      &   & 25&  616& 217945& 561983523 & --\\
        \bottomrule
     \end{tabular}
  \end{subtable}
   \newline
  \rule{0pt}{2.1cm}
   \begin{subtable}[c]{0.38\textwidth}
     \subcaption{The number of symmetry classes.}
     \begin{tabular}{l@{\hskip 2.5mm}rrrrrrr}
        \toprule
        $d\backslash n$ & 4 & 5 & 6 & 7 & 8 & 9 & 10\\
        \midrule
        2&   1 & 1& 2& 2&  3&   3&   4\\
        3&     & 1& 3& 5& 12&  47& 287\\
        4&     &  & 2& 5& 43&3147& --\\
        \bottomrule
     \end{tabular}
  \end{subtable}
\end{table}

\bigskip
\medskip
\noindent {\bf Acknowledgements.}\ I am indebted to Michael Joswig and Georg Loho for various helpful suggestions.
My research is carried out in the framework of Matheon supported by Einstein Foundation Berlin (Project \enquote{MI6 - Geometry of Equilibria for Shortest Path}).

\bibliographystyle{alpha}
\bibliography{References}

\end{document}

%% file: example1.tikz
\begin{tikzpicture}[scale = 0.56,
                    color = {black}]

   \draw[white] (-2,-1) -- (2,1);

   \tikzstyle{linestyle} = [line cap=round, line join=round];

  \coordinate (a0) at (4, 0);
  \coordinate (a1) at (7, 0);
  \coordinate (a2) at (4, 3);
  \coordinate (a3) at (7, 3);
  \coordinate (a4) at (5, 1);
  \draw[draw=blue, line width=1 pt, linestyle] (a0) -- (a4);
  \draw[draw=blue, line width=1 pt, linestyle] (a1) -- (a4) -- (a2);
  \draw[draw=black, line width=1.5 pt, linestyle] (a0) -- (a1) -- (a3) -- (a2) -- cycle;
  \fill[fill=red] (a0) circle (2.5 pt);
  \fill[fill=red] (a1) circle (2.5 pt);
  \fill[fill=red] (a2) circle (2.5 pt);
  \fill[fill=red] (a3) circle (2.5 pt);
  \draw[red, line width=1 pt, fill=red] (a4) circle (2.5 pt);

  \coordinate (b0) at (0, 0);
  \coordinate (b1) at (3, 0);
  \coordinate (b2) at (0, 3);
  \coordinate (b3) at (3, 3);
  \coordinate (b4) at (1, 1);
  \draw[draw=blue, line width=1 pt, linestyle] (b0) -- (b3);
  \draw[draw=blue, line width=1 pt, linestyle] (b1) -- (b4) -- (b2);
  \draw[draw=black, line width=1.5 pt, linestyle] (b0) -- (b1) -- (b3) -- (b2) -- cycle;
  \fill[fill=red] (b0) circle (2.5 pt);
  \fill[fill=red] (b1) circle (2.5 pt);
  \fill[fill=red] (b2) circle (2.5 pt);
  \fill[fill=red] (b3) circle (2.5 pt);
  \draw[red, line width=1 pt, fill=red] (b4) circle (2.5 pt);

  \coordinate (c0) at (8, 0);
  \coordinate (c1) at (11, 0);
  \coordinate (c2) at ( 8, 3);
  \coordinate (c3) at (11, 3);
  \coordinate (c4) at (9,1);
  \draw[draw=blue, line width=1 pt, linestyle] (c0) -- (c4);
  \draw[draw=blue, line width=1 pt, linestyle] (c1) -- (c4) -- (c2) -- cycle;
  \draw[draw=black, line width=1.5 pt, linestyle] (c0) -- (c1) -- (c3) -- (c2) -- cycle;
  \fill[fill=red] (c0) circle (2.5 pt);
  \fill[fill=red] (c1) circle (2.5 pt);
  \fill[fill=red] (c2) circle (2.5 pt);
  \fill[fill=red] (c3) circle (2.5 pt);
  \draw[red, line width=1 pt, fill=red] (c4) circle (2.5 pt);

  \coordinate (d0) at ( 4, 4);
  \coordinate (d1) at ( 7, 4);
  \coordinate (d2) at ( 4, 7);
  \coordinate (d3) at ( 7, 7);
  \coordinate (d4) at ( 5, 5);
  \draw[draw=black, line width=1.5 pt, linestyle] (d0) -- (d1) -- (d3) -- (d2) -- cycle;
  \fill[fill=red] (d0) circle (2.5 pt);
  \fill[fill=red] (d1) circle (2.5 pt);
  \fill[fill=red] (d2) circle (2.5 pt);
  \fill[fill=red] (d3) circle (2.5 pt);
  \draw[red, line width=1 pt, fill=red] (d4) circle (2.5 pt);

  \coordinate (e0) at ( 0, 4);
  \coordinate (e1) at ( 3, 4);
  \coordinate (e2) at ( 0, 7);
  \coordinate (e3) at ( 3, 7);
  \coordinate (e4) at ( 1, 5);
  \draw[draw=blue, line width=1 pt, linestyle] (e0) -- (e3);
  \draw[draw=black, line width=1.5 pt, linestyle] (e0) -- (e1) -- (e3) -- (e2) -- cycle;
  \fill[fill=red] (e0) circle (2.5 pt);
  \fill[fill=red] (e1) circle (2.5 pt);
  \fill[fill=red] (e2) circle (2.5 pt);
  \fill[fill=red] (e3) circle (2.5 pt);
  \draw[red, line width=1 pt, fill=red] (e4) circle (2.5 pt);

  \coordinate (f0) at ( 8, 4);
  \coordinate (f1) at (11, 4);
  \coordinate (f2) at ( 8, 7);
  \coordinate (f3) at (11, 7);
  \coordinate (f4) at ( 9, 5);
  \draw[draw=blue, line width=1 pt, linestyle] (f1) -- (f2);
  \draw[draw=black, line width=1.5 pt, linestyle] (f0) -- (f1) -- (f3) -- (f2) -- cycle;
  \fill[fill=red] (f0) circle (2.5 pt);
  \fill[fill=red] (f1) circle (2.5 pt);
  \fill[fill=red] (f2) circle (2.5 pt);
  \fill[fill=red] (f3) circle (2.5 pt);
  \draw[red, line width=1 pt, fill=red] (f4) circle (2.5 pt);

  \coordinate (g0) at ( 4, 8);
  \coordinate (g1) at ( 7, 8);
  \coordinate (g2) at ( 4,11);
  \coordinate (g3) at ( 7,11);
  \coordinate (g4) at ( 5, 9);
  \draw[draw=black, line width=1.5 pt, linestyle] (g0) -- (g1) -- (g3) -- (g2) -- cycle;
  \fill[fill=red] (g0) circle (2.5 pt);
  \fill[fill=red] (g1) circle (2.5 pt);
  \fill[fill=red] (g2) circle (2.5 pt);
  \fill[fill=red] (g3) circle (2.5 pt);
  \draw[red, line width=1 pt, fill=white, draw opacity=.4] (g4) circle (2.5 pt);

  \coordinate (h0) at ( 0, 8);
  \coordinate (h1) at ( 3, 8);
  \coordinate (h2) at ( 0,11);
  \coordinate (h3) at ( 3,11);
  \coordinate (h4) at ( 1, 9);
  \draw[draw=blue, line width=1 pt, linestyle] (h0) -- (h3);
  \draw[draw=black, line width=1.5 pt, linestyle] (h0) -- (h1) -- (h3) -- (h2) -- cycle;
  \fill[fill=red] (h0) circle (2.5 pt);
  \fill[fill=red] (h1) circle (2.5 pt);
  \fill[fill=red] (h2) circle (2.5 pt);
  \fill[fill=red] (h3) circle (2.5 pt);
  \draw[red, line width=1 pt, fill=white, fill opacity=0, draw opacity=.4] (h4) circle (2.5 pt);

  \coordinate (i0) at ( 8, 8);
  \coordinate (i1) at (11, 8);
  \coordinate (i2) at ( 8,11);
  \coordinate (i3) at (11,11);
  \coordinate (i4) at ( 9, 9);
  \draw[draw=blue, line width=1 pt, linestyle] (i1) -- (i2);
  \draw[draw=black, line width=1.5 pt, linestyle] (i0) -- (i1) -- (i3) -- (i2) -- cycle;
  \fill[fill=red] (i0) circle (2.5 pt);
  \fill[fill=red] (i1) circle (2.5 pt);
  \fill[fill=red] (i2) circle (2.5 pt);
  \fill[fill=red] (i3) circle (2.5 pt);
  \draw[red, line width=1 pt, fill=white, draw opacity=.4] (i4) circle (2.5 pt);



\end{tikzpicture}

%% file: nonksplit.tikz
\begin{tikzpicture}[scale = 1.2,
                    color = {black}]


  \coordinate (v0) at (0, 0);
  \coordinate (v1) at (3, 0);
  \coordinate (v2) at (3, 2);
  \coordinate (v3) at (0, 2);
  \coordinate (v4) at (1, 1);
  \coordinate (v5) at (2, 1);

 \tikzstyle{linestyle} = [line cap=round, line join=round, line width=1 pt];

  \draw[linestyle] (v0) -- (v1) -- (v2) -- (v3) -- cycle;
  \draw[blue, linestyle] (v0) -- (v4) -- (v3);
  \draw[blue, linestyle] (v1) -- (v5) -- (v2);
  \draw[blue, linestyle] (v4) -- (v5);

   \draw[red] (v0) circle (1 pt);
   \fill[fill=red] (v1) circle (1 pt);
   \fill[fill=red] (v2) circle (1 pt);
   \fill[fill=red] (v3) circle (1 pt);
   \fill[fill=red] (v4) circle (1 pt);
   \fill[fill=red] (v5) circle (1 pt);

\end{tikzpicture}

%% file: oct.tikz
\begin{tikzpicture}[x  = {(0.9cm,-0.076cm)},
                    y  = {(-0.06cm,0.95cm)},
                    z  = {(-0.44cm,-0.29cm)},
                    scale = 1.85,
                    color = {black}]

  \draw[white] (-2.2,-1.35) -- (2,1.35);

  \coordinate (v0) at (0, 0, 0);
  \coordinate (v1) at ( 1, 0, 0);
  \coordinate (v2) at (-1, 0, 0);
  \coordinate (v3) at (0, 1, 0);
  \coordinate (v4) at (0,-1, 0);
  \coordinate (v5) at (0, 0, 1);
  \coordinate (v6) at (0, 0,-1);
 
 \tikzstyle{linestyle} = [preaction={draw=white, line cap=round, line width=1.5 pt}, line cap=round, line join=round];

  \draw[line width=1 pt, linestyle] (v3) -- (v6) -- (v4);
  \draw[draw=black, line width=1.5 pt, linestyle, preaction={draw=white, line cap=round, line width=1.8 pt}] (v1) -- (v4) -- (v2);
  \draw[draw=blue, dashed, line width=0.5 pt, linestyle] (v5) -- (v6);
  \draw[fill=blue, fill opacity=0.3, line width=1 pt, linestyle] (v1) -- (v5) -- (v2) -- (v6) -- cycle;
  \draw[draw=black, line width=1.5 pt, linestyle] (v1) -- (v3) -- (v2);
  \draw[line width=1 pt, linestyle] (v3) -- (v5) -- (v4);

   \draw[red, line width=0.6 pt, fill=blue, fill opacity=0] (v0) circle (1 pt);
   \fill[fill=red] (v1) circle (1 pt);
   \fill[fill=red] (v2) circle (1 pt);
   \fill[fill=red] (v3) circle (1 pt);
   \fill[fill=red] (v4) circle (1 pt);
   \fill[fill=black] (v5) circle (1 pt);
   \node at ($(v5)-.15*(v1)$) [text=black, inner sep=0.5pt, below left, draw=none, align=left] {$e_I$};
   \fill[fill=black] (v6) circle (1 pt);

\end{tikzpicture}

%% file: vertex_figure.tikz
\begin{tikzpicture}[scale = 1.8,
                    color = {black}]

   \draw[white] (-2,-1) -- (2,1);
   \fill[fill=green] (0,-1.3) circle (0 pt);

  \coordinate (v0) at (0, 0);
  \coordinate (v1) at ( 1, 0);
  \coordinate (v2) at (-1, 0);
  \coordinate (v3) at (0, 1);
  \coordinate (v4) at (0,-1);
 
 \tikzstyle{linestyle} = [line cap=round, line join=round];

  \draw[draw=blue, line width=1 pt, linestyle] (v1) -- (v2);
  \draw[draw=black, line width=1.5 pt, linestyle] (v1) -- (v3) -- (v2) -- (v4) -- cycle;

   \draw[red, line width=1 pt, fill=red] (v0) circle (1 pt);
   \node at ($0.15*(v4)$) [text=black, inner sep=0.5pt, right, draw=none, align=left] {$q$};
   \fill[fill=red] (v1) circle (1 pt);
   \fill[fill=red] (v2) circle (1 pt);
   \fill[fill=red] (v3) circle (1 pt);
   \fill[fill=red] (v4) circle (1 pt);

\end{tikzpicture}

%% file: troppointconf.tikz
\begin{tikzpicture}[scale = 0.8,
                    color = {black}]

  \tikzstyle{linestyle} = [line cap=round, line join=roundi,line width=1 pt];
  \tikzstyle{extremal} = [fill=black];
  \tikzstyle{inner} = [black, line width=1 pt,fill=white];

  \coordinate (pa) at (8,0);
  \coordinate (pb) at (0,3);
  \coordinate (pc) at (4,0);
  \coordinate (pd) at (4,3);
  \coordinate (pe) at (0,0);
  \coordinate (pf) at (7,3);
  \coordinate (pi) at (8,6);

  \coordinate (a0) at ($(0, 0)+(pa)$);
  \coordinate (a1) at ($(1, 0)+(pa)$);
  \coordinate (a2) at ($(3, 1)+(pa)$);
  \coordinate (a3) at ($(1, 2)+(pa)$);
  \coordinate (a4) at ($(3, 2)+(pa)$);
  \coordinate (a5) at ($(2, 3)+(pa)$);
  \coordinate (ai) at ($(1,1)+(pa)$);
  \coordinate (aj) at ($(2,1)+(pa)$);
  \coordinate (ak) at ($(2,2)+(pa)$);
  \draw[linestyle] (ak) -- (a0) -- (a1) -- (a3) -- (a5) -- (aj);
  \draw[linestyle] (a3) -- (a4) -- (a2) -- (ai);
  \draw[linestyle] (a1) -- (a4);
  \fill[extremal] (a0) circle (2 pt);
  \fill[extremal] (a1) circle (2 pt);
  \fill[extremal] (a2) circle (2 pt);
  \fill[extremal] (a3) circle (2 pt);
  \fill[extremal] (a4) circle (2 pt);
  \fill[extremal] (a5) circle (2 pt);
  \draw[inner] (ai) circle (2 pt);
  \draw[inner] (aj) circle (2 pt);
  \draw[inner] (ak) circle (2 pt);

  \coordinate (b0) at ($(0, 0)+(pb)$);
  \coordinate (b1) at ($(1, 0)+(pb)$);
  \coordinate (b2) at ($(3, 1)+(pb)$);
  \coordinate (b3) at ($(1, 2)+(pb)$);
  \coordinate (b4) at ($(3, 2)+(pb)$);
  \coordinate (b5) at ($(2, 2)+(pb)$);
  \coordinate (bi) at ($(1, 1)+(pb)$);
  \coordinate (bj) at ($(2, 1)+(pb)$);
  \draw[linestyle] (b5) -- (b0) -- (b1) -- (b3) -- (b5) -- (bj);
  \draw[linestyle] (b3) -- (b4) -- (b2) -- (bi);
  \draw[linestyle] (b1) -- (b4);
  \fill[extremal] (b0) circle (2 pt);
  \fill[extremal] (b1) circle (2 pt);
  \fill[extremal] (b2) circle (2 pt);
  \fill[extremal] (b3) circle (2 pt);
  \fill[extremal] (b4) circle (2 pt);
  \fill[extremal] (b5) circle (2 pt);
  \draw[inner] (bi) circle (2 pt);
  \draw[inner] (bj) circle (2 pt);

  \coordinate (c0) at ($(0, 0)+(pc)$);
  \coordinate (c1) at ($(1, 0)+(pc)$);
  \coordinate (c2) at ($(3, 1)+(pc)$);
  \coordinate (c3) at ($(1, 1)+(pc)$);
  \coordinate (c4) at ($(3, 2)+(pc)$);
  \coordinate (c5) at ($(2, 2)+(pc)$);
  \coordinate (ci) at ($(2, 1)+(pc)$);
  \draw[linestyle] (c1)-- (c4) -- (c5) -- (c0) -- (c1) -- (c3) -- (c5) -- (ci);
  \draw[linestyle] (c4) -- (c2) -- (c3);
  \fill[extremal] (c0) circle (2 pt);
  \fill[extremal] (c1) circle (2 pt);
  \fill[extremal] (c2) circle (2 pt);
  \fill[extremal] (c3) circle (2 pt);
  \fill[extremal] (c4) circle (2 pt);
  \fill[extremal] (c5) circle (2 pt);
  \draw[inner] (ci) circle (2 pt);

  \coordinate (d0) at ($(0, 0)+(pd)$);
  \coordinate (d1) at ($(1, 0)+(pd)$);
  \coordinate (d2) at ($(2, 1)+(pd)$);
  \coordinate (d3) at ($(1, 1)+(pd)$);
  \coordinate (d4) at ($(3, 2)+(pd)$);
  \coordinate (d5) at ($(2, 2)+(pd)$);
  \draw[linestyle] (d1)-- (d4) -- (d5) -- (d0) -- (d1) -- (d3) -- (d5);
  \draw[linestyle] (d4) -- (d2) -- (d3);
  \draw[linestyle] (d2) -- (d5);
  \fill[extremal] (d0) circle (2 pt);
  \fill[extremal] (d1) circle (2 pt);
  \fill[extremal] (d2) circle (2 pt);
  \fill[extremal] (d3) circle (2 pt);
  \fill[extremal] (d4) circle (2 pt);
  \fill[extremal] (d5) circle (2 pt);

  \coordinate (e0) at ($(0, 0)+(pe)$);
  \coordinate (e1) at ($(1, 0)+(pe)$);
  \coordinate (e2) at ($(2, 1)+(pe)$);
  \coordinate (e3) at ($(1, 2)+(pe)$);
  \coordinate (e4) at ($(3, 2)+(pe)$);
  \coordinate (e5) at ($(2, 2)+(pe)$);
  \coordinate (ei) at ($(1, 1)+(pe)$);
  \draw[linestyle] (ei) -- (e2) -- (e5)-- (e0) -- (e1) -- (e4) -- (e3) -- (e1);
  \fill[extremal] (e0) circle (2 pt);
  \fill[extremal] (e1) circle (2 pt);
  \fill[extremal] (e2) circle (2 pt);
  \fill[extremal] (e3) circle (2 pt);
  \fill[extremal] (e4) circle (2 pt);
  \fill[extremal] (e5) circle (2 pt);
  \draw[inner] (ei) circle (2 pt);

  \coordinate (f0) at ($(1, 1)+(pf)$);
  \coordinate (f1) at ($(1, 0)+(pf)$);
  \coordinate (f2) at ($(2, 1)+(pf)$);
  \coordinate (f3) at ($(1, 2)+(pf)$);
  \coordinate (f4) at ($(3, 2)+(pf)$);
  \coordinate (f5) at ($(2, 2)+(pf)$);
  \draw[linestyle] (f0) -- (f2) -- (f5)-- (f0) -- (f1) -- (f4) -- (f3) -- (f1);
  \fill[extremal] (f0) circle (2 pt);
  \fill[extremal] (f1) circle (2 pt);
  \fill[extremal] (f2) circle (2 pt);
  \fill[extremal] (f3) circle (2 pt);
  \fill[extremal] (f4) circle (2 pt);
  \fill[extremal] (f5) circle (2 pt);

  \coordinate (i0) at ($(0, 0)+(pi)$);
  \coordinate (i1) at ($(2, 0)+(pi)$);
  \coordinate (i2) at ($(0, 1)+(pi)$);
  \coordinate (i3) at ($(2, 2)+(pi)$);
  \coordinate (i4) at ($(3, 1)+(pi)$);
  \coordinate (i5) at ($(3, 2)+(pi)$);
  \coordinate (ii) at ($(1, 1)+(pi)$);
  \coordinate (ij) at ($(1, 0)+(pi)$);
  \coordinate (ik) at ($(2, 1)+(pi)$);
  \draw[linestyle] (i0) -- (i1) -- (i4) -- (i2) -- (i0) -- (i3) -- (i1);
  \draw[linestyle] (i3) -- (i5) -- (i4);
  \draw[linestyle] (ij) -- (i5);
  \fill[extremal] (i0) circle (2 pt);
  \fill[extremal] (i1) circle (2 pt);
  \fill[extremal] (i2) circle (2 pt);
  \fill[extremal] (i3) circle (2 pt);
  \fill[extremal] (i4) circle (2 pt);
  \fill[extremal] (i5) circle (2 pt);
  \draw[inner] (ii) circle (2 pt);
  \draw[inner] (ij) circle (2 pt);
  \draw[inner] (ik) circle (2 pt);

  \coordinate (g0) at ($(0, 0)+(0,6)$);
  \coordinate (g1) at ($(0, 1)+(0,6)$);
  \coordinate (g2) at ($(1, 0)+(0,6)$);
  \coordinate (g3) at ($(1, 1)+(0,6)$);
  \coordinate (g4) at ($(1, 2)+(0,6)$);
  \coordinate (g5) at ($(3, 2)+(0,6)$);
  \coordinate (gi) at ($(2, 2)+(0,6)$);
  \draw[linestyle] (g0) -- (g1) -- (g4) -- (g5) -- (g2) -- (g4);
  \draw[linestyle] (g2) -- (g0) -- (gi);
  \draw[linestyle] (g1) -- (g3);
  \fill[extremal] (g0) circle (2 pt);
  \fill[extremal] (g1) circle (2 pt);
  \fill[extremal] (g2) circle (2 pt);
  \fill[extremal] (g3) circle (2 pt);%
  \fill[extremal] (g4) circle (2 pt);
  \fill[extremal] (g5) circle (2 pt);
  \draw[inner] (gi) circle (2 pt);

  \coordinate (h0) at ($(0, 0)+(4,6)$);
  \coordinate (h1) at ($(2, 0)+(4,6)$);
  \coordinate (h2) at ($(2, 1)+(4,6)$);
  \coordinate (h3) at ($(2, 2)+(4,6)$);
  \coordinate (h4) at ($(3, 1)+(4,6)$);
  \coordinate (h5) at ($(3, 2)+(4,6)$);
  \coordinate (hi) at ($(1, 1)+(4,6)$);
  \coordinate (hj) at ($(1, 0)+(4,6)$);
  \draw[linestyle] (hj) -- (h5) -- (h4) -- (h1) -- (h0) -- (hi) -- (h4);
  \draw[linestyle] (h1) -- (h3) -- (h5);
  \draw[linestyle] (hi) -- (h3);
  \fill[extremal] (h0) circle (2 pt);
  \fill[extremal] (h1) circle (2 pt);
  \fill[extremal] (h2) circle (2 pt);%
  \fill[extremal] (h3) circle (2 pt);
  \fill[extremal] (h4) circle (2 pt);
  \fill[extremal] (h5) circle (2 pt);
  \draw[inner] (hi) circle (2 pt);
  \draw[inner] (hj) circle (2 pt);

\end{tikzpicture}